%% file: phase-transitions-of-nonsingular-Bernoulli-actions.tex
\documentclass[a4paper,11pt]{article}
\usepackage{amsmath,amsthm,amssymb,enumitem,xcolor,outlines}
\usepackage{graphicx}
\usepackage[hang]{footmisc}
\usepackage{float}
\makeatletter
\newcommand{\neutralize}[1]{\expandafter\let\csname c@#1\endcsname\count@}
\makeatother

\usepackage[nosort,nocompress,noadjust]{cite}

\usepackage[bookmarks=false,hyperfootnotes=false,colorlinks,
    linkcolor={red!60!black},
    citecolor={blue!50!black},
    urlcolor={blue!80!black}]{hyperref}

\renewcommand{\eqref}[1]{\hyperref[#1]{(\ref{#1})}}

\newlist{enumlist}{enumerate}{2}
\setlist[enumlist,1]{labelindent=0cm,label=\arabic*.,ref=\arabic*,labelwidth=2.5ex,labelsep=0.5ex,leftmargin=3ex,align=left,topsep=0.5ex,itemsep=1ex,parsep=1ex}
\setlist[enumlist,2]{labelindent=0cm,label=\theenumlisti.\arabic*.,ref=\arabic*,labelwidth=5ex,labelsep=0.5ex,leftmargin=5.5ex,align=left,topsep=0.5ex,itemsep=1ex,parsep=1ex}
\newlist{itemlist}{itemize}{1}
\setlist[itemlist]{labelindent=0cm,label=$\bullet$,labelwidth=2.5ex,labelsep=0.5ex,leftmargin=3ex,align=left,topsep=0.5ex,itemsep=1ex,parsep=1ex}

\numberwithin{equation}{section}

{\theoremstyle{definition}\newtheorem{definition}{Definition}[section]
\newtheorem*{definition*}{Definition}

\newtheorem{remark}[definition]{Remark}

\newtheorem*{example*}{Example}
\newtheorem*{examples*}{Examples}}

\newtheorem{proposition}[definition]{Proposition}
\newtheorem{lemma}[definition]{Lemma}
\newtheorem{theorem}[definition]{Theorem}

\newtheorem{letterthm}{Theorem}

{\theoremstyle{definition}}

\newcommand*\xbar[1]{%
  \hbox{%
    \vbox{%
      \hrule height 0.5pt 
      \kern0.5ex
      \hbox{%
        \kern-0.1em
        \ensuremath{#1}%
        \kern-0.1em
      }%
    }%
  }%
}

\newcommand{\C}{\mathbb{C}}

\newcommand{\Z}{\mathbb{Z}}
\newcommand{\vphi}{\varphi}

\newcommand{\id}{\mathord{\text{\rm id}}}

\newcommand{\ovt}{\mathbin{\overline{\otimes}}}

\newcommand{\R}{\mathbb{R}}

\newcommand{\actson}{\curvearrowright}

\newcommand{\Aut}{\operatorname{Aut}}

\begin{document}

\begin{center}
{\boldmath\LARGE\bf Phase transitions for nonsingular Bernoulli actions}

\bigskip

{\sc by Tey Berendschot\footnote{\noindent KU~Leuven, Department of Mathematics, Leuven (Belgium).\\ E-mail: tey.berendschot@kuleuven.be\\
T.B.\ is supported by a PhD fellowship fundamental research of the Research Foundation Flanders.}}
\end{center}

\vspace{0.5ex}

\begin{abstract}\noindent
Inspired by the phase transition results for nonsingular Gaussian actions introduced in \cite{AIM19}, we prove several phase transition results for nonsingular Bernoulli actions. For generalized Bernoulli actions arising from groups acting on trees, we are able to give a very precise description of their ergodic theoretical properties in terms of the Poincaré exponent of the group. 
\end{abstract}

\vspace{0.5ex}

\input{introduction.tex}
\input{preliminaries.tex}
\input{phase.tex}
\input{trees.tex}

\end{document}

%% file: introduction.tex
\section{Introduction}\label{sec:introduction}

When $G$ is a countable infinite group and $(X_0,\mu_0)$ is a nontrivial standard probability space, the probability measure preserving (pmp) action 
\begin{align*}
G\actson (X_0,\mu_0)^{G}:\;\;\;\; (g\cdot x)_h=x_{g^{-1}h}
\end{align*}
is called a \emph{Bernoulli action}. Probability measure preserving Bernoulli actions are among the most well studied objects in ergodic theory and they play an important role in operator algebras \cite{Pop03, Pop06, Ioa10}. When we consider a family of probability measures $(\mu_g)_{g\in G}$ on the base space $X_0$ that need not all be equal, the Bernoulli action 
\begin{align}\label{eq:nonsingular Bernoulli action}
G\actson (X,\mu)=\prod_{g\in G}(X_0,\mu_g)
\end{align} 
is in general not measure preserving anymore. Instead, we are interested in the case when $G\actson (X,\mu)$ is \emph{nonsingular}, i.e. the group $G$ preserves the \emph{measure class} of $\mu$. By Kakutani's criterion for equivalence of infinite product measures the Bernoulli action \eqref{eq:nonsingular Bernoulli action} is nonsingular if and only if $\mu_h\sim \mu_g$ for every $h,g\in G$ and 
\begin{align}\label{eq:Kakutani}
\sum_{h\in G}H^2(\mu_h,\mu_{gh})<+\infty, \text{ for every } g\in G.
\end{align}
Here $H^2(\mu_h,\mu_{gh})$ denotes the \emph{Hellinger distance} between $\mu_h$ and $\mu_{gh}$, see \eqref{eq:Hellinger distance}.

It is well known that a pmp Bernoulli action $G\actson (X_0,\mu_0)^{G}$ is mixing. In particular it is ergodic and conservative. However, for nonsingular Bernoulli actions, determining conservativeness and ergodicity is much more difficult, see for instance \cite{VW17,Dan18, Kos18, BKV19}. 

Besides nonsingular Bernoulli actions, another interesting class of nonsingular group actions comes from the \emph{Gaussian construction}, as introduced in \cite{AIM19}. If $\pi\colon G\rightarrow \mathcal{O}(\mathcal{H})$ is an orthogonal representation of a locally compact second countable (lcsc) group on a real Hilbert space $\mathcal{H}$, and if $c\colon G\rightarrow \mathcal{H}$ is a 1-cocycle for the representation $\pi$, then the assignment 
\begin{align}\label{eq:affine isometric action}
\alpha_g(\xi)=\pi_g(\xi)+c(g)
\end{align}
defines an \emph{affine isometric action} $\alpha\colon G\actson \mathcal{H}$. To any affine isometric action $\alpha\colon G\actson \mathcal{H}$ Arano, Isono and Marrakchi associated a nonsingular group action $\widehat{\alpha}\colon G\actson \widehat{\mathcal{H}}$, where $\widehat{\mathcal{H}}$ is the Gaussian probability space associated to $\mathcal{H}$. When $\alpha\colon G\actson \mathcal{H}$ is actually an orthogonal representation, this construction is well-established and the resulting Gaussian action is pmp. As explained below \cite[Theorem D]{BV20}, if $G$ is a countable infinite group and $\pi\colon G\rightarrow \ell^2(G)$ is the left regular representation, the affine isometric representation \eqref{eq:affine isometric action} gives rise to a nonsingular action that is conjugate with the Bernoulli action $G\actson \prod_{g\in G}(\R,\nu_{F(g)})$, where $F\colon G\rightarrow \R$ is such that $c_g(h)=F(g^{-1}h)-F(h)$, and $\nu_{F(g)}$ denotes the Gaussian probability measure with mean $F(g)$ and variance $1$.

By scaling the 1-cocycle $c\colon G\rightarrow \mathcal{H}$ with a parameter $t\in [0,+\infty)$ we get a one-parameter family of nonsingular actions $\widehat{\alpha}^{t}\colon G\actson \widehat{\mathcal{H}}^{t}$ associated to the affine isometric actions $\alpha^{t}\colon G\actson \mathcal{H}$, given by $\alpha^t_g(\xi)=\pi_g(\xi)+tc(g)$. Arano, Isono and Marrakchi showed that there exists a $t_{diss}\in [0,+\infty)$ such that $\widehat{\alpha}^t$ is dissipative up to compact stabilizers for every $t>t_{diss}$ and infinitely recurrent for every $t<t_{diss}$ (see Section \ref{sec:preliminaries} for terminology). 

Inspired by the results obtained in \cite{AIM19}, we study a similar phase transition framework, but in the setting of nonsingular Bernoulli actions. Such a phase transition framework for nonsingular Bernoulli actions was already considered by Kosloff and Soo in \cite{KS20}. They showed the following phase transition result for the family of nonsingular Bernoulli actions of $G=\Z$ with base space $X_0=\{0,1\}$ that was introduced in \cite[Corollary 6.3]{VW17}: for every $t\in [0,+\infty)$ consider the family of measures $(\mu_n^t)_{n\in \Z}$ given by
\begin{align*}
\mu_n^t(0)=\begin{cases}1/2 &\text{ if } n\leq 4t^2\\
1/2+t/\sqrt{n}&\text{ if  }n >4t^2\end{cases}.
\end{align*}
Then $\Z\actson (X,\mu_t)=\prod_{n\in \Z}(\{0,1\},\mu_n^t)$ is nonsingular for every $t\in [0,+\infty)$. Kosloff and Soo showed that there exists a $t_1\in (1/6,+\infty)$ such that $ \Z\actson (X,\mu_t)$ is conservative for every $t<t_1$ and dissipative for every $t>t_1$ \cite[Theorem 3]{KS20}. In \cite[Example D]{DKR20} the authors describe a family of \emph{nonsingular Poisson suspensions} for which a similar phase transition occurs. These examples arise from dissipative essentially free actions of $\Z$, and thus they are nonsingular Bernoulli actions. We generalize the phase transition result from \cite{KS20} to arbitrary nonsingular Bernoulli actions as follows.

Suppose that $G$ is a countable infinite group and let $(\mu_g)_{g\in G}$ be a family of equivalent probability measure on a standard Borel space $X_0$. Let $\nu$ also be a probability measure on $X_0$. For every $t\in [0,1]$ we consider the family of equivalent probability measures $(\mu_g^t)_{g\in G}$ that are defined by
\begin{align}\label{eq:convex combination}
\mu_g^t=(1-t)\nu+t\mu_g.
\end{align}
Our first main result is that in this setting there is a phase transition phenomenon.

\begin{letterthm}\label{thm:phase transition}
Let $G$ be a countable infinite group and assume that the Bernoulli action $G\actson (X,\mu_1)=\prod_{g\in G}(X_0,\mu_g)$ is nonsingular. Let $\nu\sim \mu_e$ be a probability measure on $X_0$ and for every $t\in[0,1]$ consider the family $(\mu_g^t)_{g\in G}$ of equivalent probability measures given by \eqref{eq:convex combination}. Then the Bernoulli
action 
\begin{align*}
G\actson (X,\mu_t)=\prod_{g\in G}(X_0,\mu_g^t)
\end{align*}
is nonsingular for every $t\in [0,1]$ and there exists a $t_1\in [0,1]$ such that $G\actson (X,\mu_t)$ is weakly mixing for every $t<t_1$ and dissipative for every $t>t_1$. 
\end{letterthm}

Suppose that $G$ is a nonamenable countable infinite group. Recall that for any standard probability space $(X_0,\mu_0)$, the pmp Bernoulli action $G\actson (X_0,\mu_0)^{G}$ is strongly ergodic. Consider again the family of probability measures $(\mu_g^t)_{g\in G}$ given by \eqref{eq:convex combination}. In Theorem \ref{thm:strong ergodicity} below we prove that for $t$ close enough to $0$, the resulting nonsingular Bernoulli action is strongly ergodic. This is inspired by \cite[Theorem 7.20]{AIM19} and \cite[Theorem 5.1]{MV20}, which state similar results for nonsingular Gaussian actions. 

\begin{letterthm}\label{thm:strong ergodicity}
Let $G$ be a countable infinite nonamenable group and suppose that the Bernoulli action  $G\actson (X,\mu_1)=\prod_{g\in G}(X_0,\mu_g)$ is nonsingular. Let $\nu\sim \mu_e$ be a probability measure on $X_0$ and for every $t\in [0,1]$ consider the family $(\mu_g^t)_{g\in G}$ of equivalent probability measures given by \eqref{eq:convex combination}. Then there exists a $t_0\in (0,1]$ such that $G\actson (X,\mu_t)=\prod_{g\in G}(X_0,\mu_g^t)$ is strongly ergodic for every $t<t_0$. 
\end{letterthm}

Although we can prove a phase transition result in large generality, it remains very challenging to compute the critical value $t_1$. However, when $G\subset \Aut(T)$, for some locally finite tree $T$, following \cite[Section 10]{AIM19}, we can construct \emph{generalized Bernoulli actions} of which we can determine the conservativeness behaviour very precisely. To put this result into perspective, let us first explain briefly the construction from \cite[Section 10]{AIM19}.

For a locally finite tree $T$, let $\Omega(T)$ denote the set of orientations on $T$. Let $p\in (0,1)$ and fix a root $\rho\in T$. Define a probability measure $\mu_p$ on $\Omega(T)$ by orienting an edge towards $\rho$ with probability $p$ and away from $\rho $ with probability $1-p$. If $G\subset \Aut(T)$ is a subgroup, then we naturally obtain a nonsingular action $G\actson (\Omega(T),\mu_p)$. Up to equivalence of measures, the measure $\mu_p$ does not depend on the choice of root $\rho\in T$. The \emph{Poincaré exponent} of  $G\subset \Aut(T)$ is defined as 
\begin{align}\label{eq:Poincare exponent intro}
\delta(G\actson T)=\inf\{s>0 \text{ for which }\sum_{w\in G\cdot v}\exp(-s d(v,w))<+\infty\},
\end{align}
where $v\in V(T)$ is any vertex of $T$. In \cite[Theorem 10.4]{AIM19} Arano, Isono and Marrakchi showed that if $G\subset \Aut(T)$ is a closed nonelementary subgroup, the action $G\actson (\Omega(T),\mu_p)$ is dissipative up to compact stabilizers if $2\sqrt{p(1-p)}<\exp(-\delta)$ and weakly mixing if $2\sqrt{p(1-p)}>\exp(-\delta)$. This motivates the following similar construction. 

Let $E(T)\subset V(T)\times V(T)$ denote the set of \emph{oriented edges}, so that vertices $v$ and $w$ are adjacent if and only if $(v,w),(w,v)\in E(T)$. Suppose that $X_0$ is a standard Borel space and that $\mu_0,\mu_1$ are equivalent probability measures on $X_0$. Fix a root $\rho\in T$ and define a family of probability measures $(\mu_e)_{e\in E(T)}$ by 
\begin{align}\label{eq:tree measures intro}
\mu_e=\begin{cases}\mu_0 &\text{ if } e \text{ is oriented towards } \rho\\
\mu_1 &\text{ if } e \text{ is oriented away from } \rho\end{cases}.
\end{align}
Suppose that $G\subset \Aut(T)$ is a subgroup. Then the generalized Bernoulli action
\begin{align}\label{eq:tree actions}
G\actson \prod_{e\in E(T)}(X_0,\mu_e):\;\;\;\; (g\cdot x)_e=x_{g^{-1}\cdot e} 
\end{align}
is nonsingular and up to conjugacy it does not depend on the choice of root $\rho\in T$. In our next main result we generalize \cite[Theorem 10.4]{AIM19} to nonsingular actions of the from \eqref{eq:tree actions}.   
 
\begin{letterthm}\label{thm:trees}
Let $T$ be a locally finite tree with root $\rho\in T$ and let $G\subset \Aut(T)$ be a nonelementary closed subgroup with Poincaré exponent $\delta=\delta(G\actson T)$. Let $\mu_0$ and $\mu_1$ be equivalent probability measures on a standard Borel space $X_0$ and define a family of equivalent probability measures $(\mu_e)_{e\in E(T)}$ by \eqref{eq:tree measures intro}. Then the generalized Bernoulli action \eqref{eq:tree actions} is dissipative up to compact stabilizers if $1-H^2(\mu_0,\mu_1)<\exp(-\delta/2)$ and weakly mixing if $1-H^2(\mu_0,\mu_1)>\exp(-\delta/2)$.  
\end{letterthm}

{\it Acknowledgement.} I thank Stefaan Vaes for his valuable feedback during the process of writing this paper.


%% file: preliminaries.tex
\section{Preliminaries}\label{sec:preliminaries}
\subsection{Nonsingular group actions}

Let $(X,\mu), (Y,\nu)$ be standard measure spaces. A Borel map $\vphi \colon X\rightarrow Y$ is called \emph{nonsingular} if the pushforward measure $\vphi_*\mu$ is equivalent with $\nu$. If in addition there exist conull Borel sets $X_0\subset X$ and $Y_0\subset Y$ such that $\vphi\colon X_0\rightarrow Y_0$ is a bijection we say that $\vphi$ is a \emph{nonsingular isomorphism}. We write $\Aut(X,\mu)$ for the group of all nonsingular automorphisms $\vphi\colon X\rightarrow X$, where we identify two elements if they agree almost everywhere. The group $\Aut(X,\mu)$ carries a canonical Polish topology.

A nonsingular group action $G\actson (X,\mu)$ of  a locally compact second countable (lcsc) group $G$ on a standard measure space $(X,\mu)$ is a continuous group homomorphism $G\rightarrow \Aut(X,\mu)$. A nonsingular group action $G\actson (X,\mu)$ is called \emph{essentially free} if the stabilizer subgroup $G_x=\{g\in G:g\cdot x=x\}$ is trivial for a.e. $x\in X$. When $G$ is countable this is the same as the condition that $\mu(\{x\in X:g\cdot x=x\})=0$ for every $g\in G\setminus \{e\}$. We say that $G\actson (X,\mu)$ is \emph{ergodic} if every $G$-invariant Borel set $A\subset X$ satisfies $\mu(A)=0$ or $\mu(X\setminus A)=0$. A nonsingular action $G\actson (X,\mu)$ is called \emph{weakly mixing} if for any ergodic probability measure preserving (pmp) action $G\actson (Y,\nu)$ the diagonal product action $G\actson X\times Y$ is ergodic. If $G$ is not compact and $G\actson (X,\mu)$ is pmp, we say $G\actson X$ is \emph{mixing} if 
\begin{align*}
\lim_{g\rightarrow \infty}\mu(g\cdot A\cap B)=\mu(A)\mu(B) \text{ for every pair of Borel subsets } A,B\subset X.
\end{align*} 

Suppose that $G\actson (X,\mu)$ is a nonsingular action and that $\mu$ is a probability measure. A sequence of Borel subsets $A_n\subset X$ is called \emph{almost invariant} if 
\begin{align*}
\sup_{g\in K}\mu(g\cdot A_n\triangle A_n)\rightarrow 0 \text{ for every compact subset } K\subset G.
\end{align*}
The action $G\actson (X,\mu)$ is called \emph{strongly ergodic} if every almost invariant sequence $A_n\subset X$ is trivial, i.e. $\mu(A_n)(1-\mu(A_n))\rightarrow 0$. The strong ergodicity of $G\actson (X,\mu)$ only depends on the measure class of $\mu$. When $(Y,\nu)$ is a standard measure space and $\nu$ is infinite, a nonsingular action $G\actson (Y,\nu)$ is called strongly ergodic if $G\actson (Y,\nu')$ is strongly ergodic, where $\nu'$ is a probability measure that is equivalent with $\nu$.

Following \cite[Definition A.16]{AIM19}, we say that a nonsingular action $G\actson (X,\mu)$ is \emph{dissipative up to compact stabilizers} if each ergodic component is of the form $G\actson G/ K$, for a compact subgroup $K\subset G$. By \cite[Theorem A.29]{AIM19} a nonsingular action $G\actson (X,\mu)$, with $\mu(X)=1$, is dissipative up to compact stabilizers if and only if 
\begin{align*}
\int_G \frac{dg\mu}{d\mu}(x)d\lambda(g)<+\infty \text{ for a.e. } x\in X,
\end{align*}
where $\lambda$ denotes the left invariant Haar measure on $G$. We say that $G\actson (X,\mu)$ is \emph{infinitely recurrent} if for every nonnegligible subset $A\subset X$ and every compact subset $K\subset G$ there exists $g\in G\setminus K$ such that $\mu(g\cdot A\cap A)>0$. By \cite[Proposition A.28]{AIM19} and Lemma \ref{lem:infinitely recurrent} below, a nonsingular action $G\actson (X,\mu)$, with $\mu(X)=1$, is infinitely recurrent if and only if 
\begin{align*}
\int_G \frac{dg\mu}{d\mu}(x)d\lambda(g)=+\infty \text{ for a.e. } x\in X. 
\end{align*}

A nonsingular action $G\actson (X,\mu)$ is called \emph{dissipative} if it is essentially free and dissipative up to compact stabilizers. In that case there exists a standard measure space $(X_0,\mu_0)$ such that $G\actson X$ is conjugate with the action $G\actson G\times X_0: \;g\cdot(h,x)=(gh,x)$. A nonsingular action $G\actson (X,\mu)$ decomposes, uniquely up to a null set, as $G\actson D\sqcup C$, where $G\actson D$ is dissipative up to compact stabilizers and $G\actson C$ is infinitely recurrent. When $G$ is a countable group and $G\actson (X,\mu)$ is essentially free, we say $G\actson X$ is \emph{conservative} if it is infinitely recurrent.   

\begin{lemma}\label{lem:infinitely recurrent}
Suppose that $G$ is a lcsc group with left invariant Haar measure $\lambda$ and that $(X,\mu)$ is a standard probability space. Assume that $G\actson (X,\mu)$ is a nonsingular action that is infinitely recurrent. Then we have that
\begin{align*}
\int_G \frac{dg\mu}{d\mu}(x)d\lambda(g)=+\infty \text{ for a.e. } x\in X. 
\end{align*}
\end{lemma}

\begin{proof}
Note that the set 
\begin{align*}
D=\{x\in X:\int_G\frac{dg\mu}{d\mu}(x)d\lambda(g)<+\infty\}
\end{align*}
is $G$-invariant. Therefore it suffices to show that $G\actson X$ is not infinitely recurrent under the assumption that $D$ has full measure. 

Let $\pi\colon (X,\mu)\rightarrow (Y,\nu)$ be the projection onto the space of ergodic components of $G\actson X$. Then there is a conull Borel subset $Y_0\subset Y$ and a Borel map $\theta\colon Y_0\rightarrow X$ such that $(\pi\circ \theta)(y)=y$ for every $y\in Y_0$. 

Write $X_y=\pi^{-1}(\{y\})$. By \cite[Theorem A.29]{AIM19}, for a.e. $y\in Y$ there exists a compact subgroup $K_y\subset G$ such that $G\actson X_y$ is conjugate with $G\actson G/ K_y$. Let $G_n\subset G$ be an increasing sequence of compact subsets of $G$ such that $\bigcup_{n\geq 1}\overset{\circ}{G_n}=G$. 
For every $x\in X$, write $G_x=\{g\in G:g\cdot x=x\}$ for the stabilizer subgroup of $x$. Using an argument as in \cite[Lemma 10]{MRV11} one shows that for each $n\geq 1$ the set $\{x\in X:G_x\subset G_n\}$ is Borel. Thus, for every $n\geq 1$, the set 
\begin{align*}
U_n=\{y\in Y_0:K_y\subset G_n\}=\{y\in Y_0:G_{\theta(y)}\subset G_n\}
\end{align*}
is a Borel subset of $Y$ and we have that $\nu(\bigcup_{n\geq 1 } U_n)=1$. Therefore, the sets 
\begin{align*}
A_n=\{g\cdot \theta(y):g\in G_n, \;y\in U_{n}\}
\end{align*}
are analytic and exhaust $X$ up to a set of measure zero. So there exists an $n_0\in \mathbb{N}$ and a nonnegligible Borel set $B\subset A_{n_0}$.
Suppose that $h\in G$ is such that $h\cdot B\cap B\neq \emptyset$, then there exists $y\in U_{n_0}$ and $g_1,g_2\in G_{n_0}$ such that $hg_1\cdot \theta(y)=g_2\cdot \theta(y)$, and we get that $h\in G_{n_0}K_yG_{n_0}^{-1}\subset G_{n_0}G_{n_0}G_{n_0}^{-1}$. In other words, for $h\in G$ outside the compact set $G_{n_0}G_{n_0}G_{n_0}^{-1}$ we have that $\mu(h\cdot B \cap B)=0$, so that $G\actson X$ is not infinitely recurrent.
\end{proof}

We will frequently use the following result by Schmidt--Walters. Suppose that $G\actson (X,\mu)$ is a nonsingular action that is infinitely recurrent and suppose that $G\actson (Y,\nu)$ is pmp and mixing. Then by \cite[Theorem 2.3]{SW81} we have that
\begin{align*}
L^{\infty}(X\times Y)^{G}=L^{\infty}(X)^{G}\ovt 1,
\end{align*}
where $G\actson X\times Y$ acts diagonally. Although \cite[Theorem 2.3]{SW81} demands proper ergodicity of the action $G\actson (X,\mu)$, the infinite recurrence assumption is sufficient as remarked in \cite[Remark 7.4]{AIM19}.

\subsection{The Maharam extension and crossed products}
Let $(X,\mu)$ be a standard measure space. For any nonsingular automorphism $\vphi\in \Aut(X,\mu)$, we define its \emph{Maharam extension} by
\begin{align*}
\widetilde{\vphi}\colon X\times \R\rightarrow X\times \R:\;\;\;\; \widetilde{\vphi}(x,t)=(\vphi(x),t+\log(d\vphi^{-1}\mu/d\mu) (x)).
\end{align*}
Then $\widetilde{\vphi}$ preserves the infinite measure $\mu\times \exp(-t)dt$. The assignment $\vphi\mapsto\widetilde{\vphi}$ is a continuous group homomorphism from $\Aut(X)$ to $\Aut(X\times \R)$. Thus for each nonsingular group action $G\actson (X,\mu)$, by composing with this map, we obtain a nonsingular group action $G\actson X\times \R$, which we call the \emph{Maharam extension of $G\actson X$}. If $G\actson X$ is a nonsingular group action, the translation action $\R\actson X\times \R$ in the second component commutes with the Maraham extension $G\actson X\times \R$. Therefore, we get a well defined action $\R\actson L^{\infty}(X\times \R)^{G}$, which is the \emph{Krieger flow} associated to the action $G\actson X$. The Krieger flow is given by $\R\actson \R$ if and only if there exists a $G$-invariant $\sigma$-finite measure $\nu$ on $X$ that is equivalent with $\mu$.

Suppose that $M\subset B(\mathcal{H})$ is a von Neumann algebra represented on the Hilbert space $\mathcal{H}$ and that $\alpha\colon G\actson M$ is a continuous action on $M$ of a lcsc group $G$. Then the \emph{crossed product von Neumann algebra} $M\rtimes_{\alpha} G\subset B(L^2(G,\mathcal{H}))$ is the von Neumann algebra generated by the operators $\{\pi(x)\}_{x\in M}$ and $\{u_h\}_{h\in G}$ acting on $\xi\in L^2(G,\mathcal{H})$ as 
\begin{align*}
(\pi(x)\xi)(g)=\alpha_{g^{-1}}(x)\xi(g),\hspace{1cm}(u_h\xi)(g)=\xi(h^{-1}g).
\end{align*} 
In particular, if $G\actson (X,\mu)$ is a nonsingular group action, the crossed product $L^{\infty}(X)\rtimes G\subset B(L^2(G\times X))$ is the von Neumann algebra generated by the operators 
\begin{align*}
(\pi(H)\xi)(g,x)=H(g\cdot x)\xi(g,x), \hspace{1cm} (u_h\xi)(g,x)=\xi(h^{-1}g,x),
\end{align*}
for $H\in L^{\infty}(X)$ and $h\in G$. If $G\actson X$ is nonsingular essentially free and ergodic, then $L^{\infty}(X)\rtimes G$ is a factor. When moreover $G$ is a unimodular group, the Krieger flow of $G\actson X$ equals the flow of weights of the crossed product von Neumann algebra $L^{\infty}(X)\rtimes G$. For nonunimodular groups this is not necessarily true, motivating the following definition. 

\begin{definition}\label{def:modular Maharam}
Let $G$ be an lcsc group with modular function $\Delta\colon G\rightarrow\R_{>0}$. Let $\lambda$ denote the Lebesgue measure on $\R$. Suppose that $\alpha\colon  G\actson (X,\mu)$ is a nonsingular action. We define the \emph{modular Maharam extension} of $G\actson X$ as the nonsingular action 
\begin{align*}
\beta\colon G\actson (X\times \R, \mu\times \lambda): \;\;\;\; g\cdot(x,t)=(g\cdot x, t+\log(\Delta(g))+\log(dg^{-1}\mu/d\mu)(x)).
\end{align*}
Let $L^{\infty}(X\times \R)^{\beta}$ denote the subalgebra of $\beta$-invariant elements. We define the \emph{flow of weights} associated to $G\actson X$ as the translation action  $\R\actson L^{\infty}(X\times \R)^{\beta}: (t\cdot H)(x,s)=H(x,s-t)$.
\end{definition}

As we explain below, the flow of weights associated to an essentially free ergodic nonsingular action $G\actson X$ equals the flow of weights of the crossed product factor $L^{\infty}(X)\rtimes G$, justifying the terminology. See also \cite[Proposition 4.1]{Sa74}.

Let $\alpha\colon G\actson X$ be an essentially free ergodic nonsingular group action with modular Maharam extension $\beta\colon G\actson X\times \R$. By \cite[Proposition 1.1]{Sa74} there is a canonical normal semifinite faithful weight $\vphi$ on $L^{\infty}(X)\rtimes_{\alpha} G$ such that the modular automorphism group $\sigma^{\vphi}$ is given by
\begin{align*}
\sigma^{\vphi}_t(\pi(H))=\pi(H), \hspace{1cm} \sigma^{\vphi}_t(u_g)=\Delta(g)^{it}u_g\pi((dg^{-1}\mu/d\mu)^{it}),
\end{align*}
where $\Delta\colon G\rightarrow \R_{>0}$ denotes the modular function of $G$. 

For an element $\xi\in L^2(\R, L^2(G\times X))$ and $(g,x)\in G\times X$, write $\xi_{g,x}$ for the map given by $\xi_{g,x}(s)=\xi(s,g,x)$. Then by Fubini's theorem $\xi_{g,x}\in L^2(\R)$ for a.e. $(g,x)\in G\times X$. Let $U\colon L^2(\R, L^2(G\times X))\rightarrow L^2(G,L^2(X\times \R))$ be the unitary given on $\xi\in L^2(\R, L^2(G\times X))$ by
\begin{align*}
(U\xi)(g,x,t)=\mathcal{F}^{-1}(\xi_{g,x})(t+\log(\Delta(g))+\log(dg^{-1}\mu/d\mu)(x)),
\end{align*}
where $\mathcal{F}^{-1}\colon L^2(\R)\rightarrow L^2(\R)$ denotes the inverse Fourier transform. One checks that conjugation by $U$ induces an isomorphism 
\begin{align*}
\Psi\colon (L^{\infty}(X)\rtimes_{\alpha} G)\rtimes_{\sigma^{\vphi}}\R\rightarrow L^{\infty}(X\times \R)\rtimes_{\beta}G.
\end{align*}
Let $\kappa\colon L^{\infty}(X\times \R)\rightarrow  L^{\infty}(X\times \R)\rtimes_{\beta}G$ be the inclusion map and let $\gamma\colon \R\actson L^{\infty}(X\times \R)\rtimes_{\beta}G$ be the action given by 
\begin{align*}
\gamma_t(\kappa(H))(x,s)=\kappa(H)(x,s-t),\hspace{1cm}\gamma_t(u_g)=u_g.
\end{align*}
Then one can verify that $\Psi$ conjugates the dual action $\widehat{\sigma^{\vphi}}\colon \R\actson  (L^{\infty}(X)\rtimes_{\alpha} G)\rtimes_{\sigma^{\vphi}}\R$ and $\gamma$. Therefore we can identify the flow of weights $\R\actson \mathcal{Z}((L^{\infty}(X)\rtimes_{\alpha} G)\rtimes_{\sigma^{\vphi}}\R)$ with $\R\actson \mathcal{Z}(L^{\infty}(X\times \R)\rtimes_{\beta} G)\cong L^{\infty}(X\times \R)^{\beta}$: the flow of weights associated to $G\actson X$.

\begin{remark}\label{rem:type classificication}
It will be useful to speak about the \emph{Krieger type} of a nonsingular ergodic action $G\actson X$. In light of the discussion above, we will only use this terminology for countable groups $G$, so that no confusion arises with the type of the crossed product von Neumann algebra $L^{\infty}(X)\rtimes G$. So assume that $G$ is countable and that $G\actson (X,\mu)$ is a nonsingular ergodic action. Then the Krieger flow is ergodic and we distinguish several cases. If $\nu$ is atomic, we say that $G\actson X$ is of type I. If $\nu$ is nonatomic and finite, we say that $G\actson X$ is of type II$_{1}$. If $\nu$ is nonatomic and infinite, we say $G\actson X$ is of type II$_{\infty}$. If the Krieger flow is given by $\R\actson \R/\log(\lambda)\Z$ with $\lambda\in (0,1)$, we say that $G\actson X$ is of type III$_{\lambda}$. If the Krieger flow is the trivial flow $\R\actson \{\ast\}$, we say that $G\actson X$ is of type III$_{1}$. If the Krieger flow is properly ergodic, i.e. every orbit has measure zero, we say that $G\actson X$ is of type III$_{0}$.
 \end{remark}

\subsection{Nonsingular Bernoulli actions}

Suppose that $G$ is a countable infinite group and that $(\mu_g)_{g\in G}$ is a family of equivalent probability measures on a standard Borel space $X_0$. The action
\begin{align}\label{eq:Bernoulli action prelim}
G\actson (X,\mu)=\prod_{h\in G}(X_0,\mu_h):\;\;\;\; (g\cdot x)_h=x_{g^{-1}h}
\end{align} 
is called the \emph{Bernoulli action}. For two probability measures $\nu,\eta$ on a standard Borel space $Y$, the \emph{Hellinger distance} $H^2(\nu,\eta)$ is defined by
\begin{align}\label{eq:Hellinger distance}
H^2(\nu,\eta)=\frac{1}{2}\int_Y\left(\sqrt{d\nu/d\zeta}-\sqrt{d\eta/d\zeta}\right)^2d\zeta,
\end{align}
where $\zeta$ is any probability measure on $Y$ such that $\nu,\eta \prec \zeta$. By Kakutani's criterion for equivalence of infinite product measures \cite{Kak48} the Bernoulli action \eqref{eq:Bernoulli action prelim} is nonsingular if and only if 
\begin{align*}
\sum_{h\in G}H^2(\mu_{h},\mu_{gh})<+\infty, \text{ for every }  g\in G.
\end{align*}
If $(X,\mu)$ is nonatomic and the Bernoulli action \eqref{eq:Bernoulli action prelim} is nonsingular, then it is essentially free by \cite[Lemma 2.2]{BKV19}.

Suppose that $I$ is a countable infinite set and that $(\mu_i)_{i\in I}$ is a family of equivalent probability measures on a standard Borel space $X_0$. If $G$ is a lcsc group that acts on $I$, the action
\begin{align}\label{eq:generalized nonsingular Bernoulli prelim}
G\actson (X,\mu)=\prod_{i\in I}(X_0,\mu_i):\;\;\;\; (g\cdot x)_i=x_{g^{-1}\cdot i}
\end{align}
is called the \emph{generalized Bernoulli action} and it is nonsingular if and only if $\sum_{i\in I}H^2(\mu_i,\mu_{g\cdot i})<+\infty$ for every $g\in G$. When $\nu$ is a probability measure on $X_0$ such that $\mu_i=\nu$ for every $i\in I$, the generalized Bernoulli action \eqref{eq:generalized nonsingular Bernoulli prelim} is pmp and it is mixing if and only if the stabilizer subgroup $G_i=\{g\in G:g\cdot i=i\}$ is compact for every $i\in I$. In particular, if $G$ is countable infinite, the pmp Bernoulli action $G\actson (X_0,\mu_0)^{G}$ is mixing.

\subsection{Groups acting on trees}

Let $T=(V(T),E(T))$ be a locally finite tree, so that the edge set $E(T)$ is a symmetric subset of $ V(T)\times V(T)$ with the property that vertices $v,w\in V(T)$ are adjacent if and only if $(v,w),(w,v)\in E(T)$. When $T$ is clear from the context, we will write $E$ instead of $E(T)$. Also we will often write $T$ instead of $V(T)$ for the vertex set. For any two vertices $v,w\in T$ let $[v,w]$ denote the smallest subtree of $T$ that contains $v$ and $w$. The distance between vertices $v,w\in T$ is defined as $d(v,w)=|V([v,w])|-1$. Fixing a root $\rho\in T$, we define the \emph{boundary} $\partial T$ of $T$ as the collection of all infinite line segments starting at $\rho$. We equip $\partial T$ with a metric $d_\rho$ as follows. If $\omega,\omega'\in \partial T$, let $v\in T$ be the unique vertex such that $d(\rho,v)=\sup_{v\in \omega\cap \omega'}d(\rho, v)$ and define 
\begin{align*}
d_\rho(\omega,\omega')=\exp(-d(\rho,v)).
\end{align*}
Then, up to homeomorphism, the space $(\partial T, d_{\rho})$ does not depend on the chosen root $\rho\in T$. Furthermore, the Hausdorff dimension $\dim_H \partial T$ of $(\partial T, d_\rho)$ is also independent of the choice of $\rho\in T$.

Let $\Aut(T)$ denote the group of automorphisms of $T$. By \cite[Proposition 3.2]{Tit70}, if $g\in \Aut(T)$, then either
\begin{itemize}
\item $g$ fixes a vertex or interchanges a pair of vertices. In this case we say $g$ is \emph{elliptic}.
\item or there exists a bi-infinite line segment $L\subset T$, called the \emph{axis} of $g$, such that $g$ acts on $L$ by nontrivial translation. In this case we say $g$ is \emph{hyperbolic}.
\end{itemize}
We equip $\Aut(T)$ with the topology of pointwise convergence. A subgroup $G\subset \Aut(T)$ is closed with respect to this topology if and only if for every $v\in T$ the stabilizer subgroup $G_v=\{g\in G:g\cdot v= v\}$ is compact. An action of a lcsc group $G$ on $T$ is a continuous homomorphism $G\rightarrow \Aut(T)$. We say the action $G\actson T$ is \emph{cocompact} if there is a finite set $F\subset E(T)$ such that $G\cdot F=E(T)$. A subgroup $G\subset \Aut(T)$ is called \emph{nonelementary} if it does not fix any point in $T\cup \partial T$ and does not interchange any pair of points in $T\cup \partial T$. Equivalently, $G\subset \Aut(T)$ is nonelementary if there exist hyperbolic elements $h,g\in G$ with axes $L_h$ and $L_g$ such that $L_h\cap L_g$ is finite. If $G\subset \Aut(T)$ is a nonelementary closed subgroup, there exists a unique minimal $G$-invariant subtree $S\subset T$ and $G$ is compactly generated if and only if $G\actson S$ is cocompact (see \cite[Section 2]{CM11}). Recall from \eqref{eq:Poincare exponent intro} the definition of the Poincaré exponent $\delta(G\actson T)$ of a subgroup $G\subset \Aut(T)$. If $G\subset \Aut(T)$ is a closed subgroup such that $G\actson T$ is cocompact, then we have that $\delta(G\actson T)=\dim_{H}\partial T$.

%% file: phase.tex
\section{Phase transitions of nonsingular Bernoulli actions: proof of Theorems \ref{thm:phase transition} \& \ref{thm:strong ergodicity}}\label{sec:phase}

Let $G$ be a countable infinite group and let $(\mu_g)_{g\in G}$ be a family of equivalent probability measures on a standard Borel space $X_0$. Let $\nu$ also be a probability measure on $X_0$. For  $t\in [0,1]$ we define the family of probability measures 
\begin{align}\label{eq:convex combo}
\mu_g^t=(1-t)\nu+t\mu_g, \;\;\;\; g\in G.
\end{align}
We write $\mu_t$ for the infinite product measure $\mu_t=\prod_{g\in G}\mu_g^t$ on $X=\prod_{g\in G}X_0$. We prove Theorem \ref{thm:phase transition general} below, that is slightly more general than Theorem \ref{thm:phase transition}.

\begin{theorem}\label{thm:phase transition general}
Let $G$ be a countable infinite group and let $(\mu_g)_{g\in G}$ be a family of equivalent probability measures on a standard probability space $X_0$, that is not supported on a single atom. Assume that the Bernoulli action $G\actson \prod_{g\in G}(X_0,\mu_g)$ is nonsingular. Let $\nu$ also be a probability measure on $X_0$. Then for every $t\in [0,1]$, the Bernoulli action
\begin{align}\label{eq:G actson (X,mu_t)}
G\actson (X,\mu_t)=\prod_{g\in G}(X_0,(1-t)\nu+t\mu_g)
\end{align}
is nonsingular. Assume in addition that one of the following conditions hold.
\begin{enumerate}
\item $\nu\sim \mu_e$.
\item $\nu\prec \mu_e $ and $\sup_{g\in G}|\log d\mu_g/d\mu_e(x)|<+\infty$ for a.e $x\in X_0$.
\end{enumerate}
Then there exists a $t_1\in [0,1]$ such that $G\actson (X,\mu_t)$ is dissipative for every $t>t_1$ and weakly mixing for every $t<t_1$.
\end{theorem}

\begin{remark}
One might hope to prove a completely general phase transition result that only requires $\nu\prec \mu_e$, and  not the additional assumption that $\sup_{g\in G}|\log d\mu_g/d\mu_e(x)|<+\infty$ for a.e. $x\in X_0$. However, the following example shows that this is not possible. 

Let $G$ be any countable infinite group and let $G\actson \prod_{g\in G}(C_0,\eta_g)$ be a conservative nonsingular Bernoulli action. Note that Theorem \ref{thm:phase transition general} implies that $$G\actson \prod_{g\in G}(C_0,(1-t)\eta_e+t\eta_g)$$ is conservative for every $t<1$. Let $C_1$ be a standard Borel space and let $(\mu_g)_{g\in G}$ be a family of equivalent probability measures on $X_0=C_0\sqcup C_1$ such that $0<\sum_{g\in G}\mu_g(C_1)<+\infty$ and such that $\mu_g\big|_{C_0}=\mu_g(C_0)\eta_g$. Then the Bernoulli action $G\actson (X,\mu)=\prod_{g\in G}(X_0,\mu_g)$ is nonsingular with nonnegligible conservative part $C_0^{G}\subset G$ and dissipative part $X\setminus C_0^G$. Taking $\nu=\eta_e\prec \mu_e$, for each $t<1$ the Bernoulli action $G\actson (X,\mu_t)= \prod_{g\in G}(X_0,(1-t)\eta_e+t\mu_g)$ is constructed in the same way, by starting with the conservative Bernoulli action $G\actson \prod_{g\in G}(C_0,(1-t)\eta_e+t\eta_g)$. So for every $t\in (0,1)$ the Bernoulli action $G\actson (X,\mu_t)$ has nonnegligible conservative part and nonnegligible dissipative part.
\end{remark}

We can also prove a version of Theorem \ref{thm:strong ergodicity} in the more general setting of Theorem \ref{thm:phase transition general}.

\begin{theorem}\label{thm:strong ergodicity general}
Let $G$ be a countable infinite nonamenable group. Make the same assumptions as in Theorem \ref{thm:phase transition general} and consider the nonsingular Bernoulli actions $G\actson (X,\mu_t)$ given by \eqref{eq:G actson (X,mu_t)}. Assume moreover that 
\begin{enumerate}
\item $\nu\sim \mu_e$, or
\item $\nu\prec \mu_e$ and $\sup_{g\in G}|\log d\mu_g/d\mu_e(x)|<+\infty$ for a.e. $x\in X_0$.
\end{enumerate}
Then there exists a $t_0>0$ such that $G\actson (X,\mu_t)$ is strongly ergodic for every $t<t_0$.
\end{theorem}

\begin{proof}[Proof of Theorem \ref{thm:phase transition general}]
Assume that $G\actson (X,\mu_1)=\prod_{g\in G}(X_0,\mu_g)$ is nonsingular. For every $t\in [0,1]$, we have that
\begin{align*}
\sum_{h\in G}H^2(\mu_h^t,\mu_{gh}^t)\leq t\sum_{h\in G}H^2(\mu_h, \mu_{gh}) \text{ for every } g\in G,
\end{align*}
so that $G\actson (X,\mu_t)$ is nonsingular for every $t\in [0,1]$. The rest of the proof we divide into two steps.

\vspace{.5cm}

{\bf Claim 1.} If $G\actson (X,\mu_t)$ is conservative, then $G\actson (X,\mu_s)$ is weakly mixing for every $s<t$.

\vspace{.5cm}

{\it Proof of claim 1.} Note that for every $g\in G$ we have that 
\begin{align*}
(\mu_g^s)^r&=(1-r)\nu+r\mu_g^s=(1-r)\nu+r(1-s)\nu+rs\mu_g=\mu_g^{sr},
\end{align*}
so that $(\mu_s)_r=\mu_{sr}$. Therefore it suffices to prove that $G\actson (X,\mu_s)$ is weakly mixing for every $s<1$, assuming that $G\actson (X,\mu_1)$ is conservative. 

The claim is trivially true for $s=0$. So assume that $G\actson (X,\mu_1)$ is conservative and fix $s\in (0,1)$. Let $G\actson (Y,\eta)$ be an ergodic pmp action. Define $Y_0=X_0\times X_0\times \{0,1\}$ and define the probability measures $\lambda$ on $\{0,1\}$ by $\lambda(0)=s$. Define the map $\theta\colon Y_0\rightarrow X_0$ by 
\begin{align}\label{eq:theta map}
\theta(x,x',j)=\begin{cases}x &\text{ if } j=0\\
x' &\text{ if } j=1\end{cases}.
\end{align}
Then for every $g\in G$ we have that $\theta_*(\mu_g\times \nu\times \lambda)=\mu_g^s$. Write $Z=\{0,1\}^G$ and equip $Z$ with the probability measure $\lambda^{G}$. We identify the Bernoulli action $G\actson Y_0^{G}$ with the diagonal action $G\actson X\times X\times Z$. By applying $\theta$ in each coordinate we obtain a $G$-equivariant factor map
\begin{align}\label{eq:factor map}
\Psi\colon X\times X\times Z\rightarrow X:\;\;\;\; \Psi(x,x',z)_h=\theta(x_h,x'_h,z_h).
\end{align}
Then the map $\id_Y\times \Psi\colon Y\times X\times X\times Z\rightarrow Y\times X$ is $G$-equivariant and we have that $(\id_Y\times \Psi)_*(\eta\times \mu_1\times \mu_0\times\lambda^G)=\eta\times \mu_s$. The construction above is similar to \cite[Section 4]{KS20}.

 Take $F\in L^{\infty}(Y\times X,\eta\times \mu_s)^{G}$. Note that the diagonal action $G\actson (Y\times X,\eta\times \mu_1)$ is conservative, since $G\actson (Y,\eta)$ is pmp. The action $G\actson (X\times Z,\mu_0\times \lambda^{G})$ can be identified with a pmp Bernoulli action with base space $(X_0\times \{0,1\},\nu\times \lambda)$, so that it is mixing. By \cite[Theorem 2.3]{SW81} we have that
\begin{align*}
L^{\infty}(Y\times X\times X\times Z,\eta\times \mu_1\times \mu_0\times \lambda^{G})^{G}=L^{\infty}(Y\times X,\eta\times \mu_1)^{G}\ovt 1\ovt 1,
\end{align*}
which implies that the assignment $(y,x,x',z)\mapsto F(y, \Psi(x,x',z))$ is essentially independent of the $x'$- and $z$-variable. Choosing a finite set of coordinates $\mathcal{F}\subset G$ and changing, for $g\in \mathcal{F}$, the value $z_g$ between $0$ and $1$, we see that $F$ is essentially independent of the $x_g$-coordinates for $g\in \mathcal{F}$. As this is true for any finite set $\mathcal{F}\subset G$, we have that $F\in L^{\infty}(Y)^{G}\ovt 1$. The action $G\actson (Y,\eta)$ is ergodic and therefore $F$ is essentially constant. We conclude that $G\actson (X,\mu_s)$ is weakly mixing.

\vspace{.5cm} 

{\bf Claim 2.} If $\nu\sim \mu_e$ and if $G\actson (X,\mu_t)$ is not dissipative, then $G\actson (X,\mu_s)$ is conservative for every $s<t$. 

\vspace{.5cm}

{\it Proof of claim 2.} Again it suffices to assume that $G\actson (X,\mu_1)$ is not dissipative and to show that $G\actson (X,\mu_s)$ is conservative for every $s<1$. 

When $s=0$, the statement is trivial, so assume that $G\actson (X,\mu_1)$ is not dissipative and fix $s\in (0,1)$. Let $C\subset X$ denote the nonnegligible conservative part of $G\actson (X,\mu_1)$. As in the proof of claim 1, write $Z=\{0,1\}^{G}$ and let $\lambda$ be the probability measure on $\{0,1\}$ given by $\lambda(0)=s$. Writing $\Psi\colon X\times X\times Z\rightarrow X$ for the $G$-equivariant map \eqref{eq:factor map}. We claim that $\Psi_*((\mu_1\times \mu_0\times \lambda^{G})\big|_{C\times X\times Z})\sim \mu_s$, so that $G\actson (X,\mu_s)$ is a factor of a conservative nonsingular action, and therefore must be conservative itself. 

As $\Psi_*(\mu_1\times \mu_0\times \lambda^{G})=\mu_s$, we have that $\Psi_*((\mu_1\times \mu_0\times \lambda^{G})\big|_{C\times X\times Z})\prec \mu_s$. Let $\mathcal{U}\subset X$ be the Borel set, uniquely determined up to a set of measure zero, such that $\Psi_*((\mu_1\times \mu_0\times \lambda^{G})\big|_{C\times X\times Z})\sim \mu_s\big|_{\mathcal{U}}$. We have to show that $\mu_s(X\setminus \mathcal{U})=0$. Fix a finite subset $\mathcal{F}\subset G$. For every $t\in [0,1]$ define 
\begin{align*}
(X_1,\gamma_1^{t})&=\prod_{g\in \mathcal{F}}(X_0,(1-t)\nu+t\mu_g),\\ 
(X_2,\gamma_2^{t})&=\prod_{g\in G \setminus \mathcal{F}}(X_0,(1-t)\nu+t\mu_g).
\end{align*}
We shall write $\gamma_1=\gamma_1^1, \gamma_2=\gamma_2^1$. Also define 
\begin{align*}
(Y_1,\zeta_1)&=\prod_{g\in \mathcal{F}}(X_0\times X_0\times \{0,1\},\mu_g\times \nu\times \lambda),\\
(Y_2,\zeta_2)&=\prod_{g\in G\setminus \mathcal{F}}(X_0\times X_0\times \{0,1\},\mu_g\times \nu\times \lambda).
\end{align*}
By applying the map \eqref{eq:theta map} in every coordinate, we get factor maps $\Psi_j\colon Y_j\rightarrow X_j$ that satisfy $(\Psi_j)_*(\zeta_j)=\gamma_j^{s}$ for $j=1,2$. Identify $X_1\times Y_2\cong X\times (X_0\times \{0,1\})^{G\setminus \mathcal{F}}$ and define the subset $C'\subset X_1\times Y_2$ by $C'=C\times (X_0\times \{0,1\})^{G\setminus \mathcal{F}}$. Let $\mathcal{U}'\subset X$ be Borel such that 
\begin{align*}
(\id_{X_1}\times \Psi_2)_*((\gamma_1\times \zeta_2)\big|_{C'})\sim (\gamma_1\times \gamma_2^{s})\big|_{\mathcal{U}'}.
\end{align*}
Identify $Y_1\times X_2\cong X\times (X_0\times \{0,1\})^{\mathcal{F}}$ and define $V\subset Y_1\times X_2$ by $V=\mathcal{U}'\times (X_0\times \{0,1\})^{\mathcal{F}}$. Then we have that
\begin{align*}
(\Psi_1\times \id_{X_2})_*((\zeta_1\times \gamma_2^s)\big|_{V})&\sim (\Psi_1\times \id_{X_2})_*(\id_{Y_1}\times \Psi_2)_*((\gamma_1\times \zeta_1)\big|_{C'}\times \nu^{\mathcal{F}}\times \lambda^{\mathcal{F}})\\
&=\Psi_*((\zeta_1\times \zeta_2)\big|_{C\times X\times Z})\sim \mu_s\big|_{\mathcal{U}}.
\end{align*}
Let $\pi\colon X_1\times X_2\rightarrow X_2$ and $\pi'\colon Y_1\times X_2\rightarrow X_2$ denote the coordinate projections. Note that by construction we have that
\begin{align} \label{eq:equivalence chain}
\pi'_*((\zeta_1\times \gamma_2^s)\big|_V) \sim \pi_*((\gamma_1\times \gamma_2^{s})\big|_{\mathcal{U}'})\sim \pi_*(\mu_s\big|_{\mathcal{U}}).
\end{align}
Let $W\subset X_2$ be Borel such that $\pi_*(\mu_s\big|_{\mathcal{U}})\sim \gamma_2^s\big|_{W}$. For every $y\in X_2$ define the Borel sets
\begin{align*}
\mathcal{U}_y=\{x\in X_1:(x,y)\in \mathcal{U}\}\;\;\;\; \text{ and } \;\;\;\;\mathcal{U}'_y=\{x\in X_1:(x,y)\in \mathcal{U}'\}.
\end{align*}
As $\pi_*((\gamma_1\times \gamma_2^s)\big|_{\mathcal{U}'})\sim \gamma_2^s\big|_{W}$, we have that 
\begin{align*}
\gamma_1(\mathcal{U}'_y)>0 \text{ for } \gamma_2^s-\text{a.e. } y\in W.
\end{align*}
The disintegration of $(\gamma_1\times \gamma_2^s)\big|_{\mathcal{U}'}$ along $\pi$ is given by $(\gamma_1\big|_{\mathcal{U}'_y})_{y\in W}$. Therefore the disintegration of $(\zeta_1\times \gamma_2^s)\big|_{V}$ along $\pi'$ is given by $(\gamma_1\big|_{\mathcal{U}'_y}\times \nu^{\mathcal{F}}\times \lambda^{\mathcal{F}})_{y\in W}$. We conclude that the disintegration of $(\Psi_1\times \id_{X_2})_*((\zeta_1\times \gamma_2^s)\big|_V)$ along $\pi$ is given by $((\Psi_1)_*(\gamma_1\big|_{\mathcal{U}'_y}\times \nu^{\mathcal{F}}\times \lambda^{\mathcal{F}}))_{y\in W}$. The disintegration of $\mu_s\big|_{\mathcal{U}}$ along $\pi$ is given by $(\gamma_2^s\big|_{\mathcal{U}_y})_{y\in W}$. Since $\mu_s\big|_{\mathcal{U}}\sim(\Psi_1\times \id_{X_2})_*((\zeta_1\times \gamma_2^s)\big|_V)$, we conclude that
\begin{align*}
(\Psi_1)_*(\gamma_1\big|_{\mathcal{U}'_y}\times \nu^{\mathcal{F}}\times \lambda^{\mathcal{F}})\sim \gamma_1^s\big|_{\mathcal{U}_y} \text{ for } \gamma_2^s-\text{a.e. } y\in W.
\end{align*} 
As $\gamma_1(\mathcal{U}'_y)>0$ for $\gamma_2^s$ - a.e. $y\in W$, and using that $\nu\sim \mu_e$, we see that
\begin{align*}
\gamma_1^{s}\sim \nu^\mathcal{F}&\sim(\Psi_1)_*((\gamma_1\times \nu^{\mathcal{F}}\times \lambda^{\mathcal{F}})\big|_{\mathcal{U}'_y\times X_0^{\mathcal{F}}\times \{1\}^{\mathcal{F}}})\\
&\prec (\Psi_1)_*(\gamma_1\big|_{\mathcal{U}'_y}\times \nu^{\mathcal{F}}\times \lambda^{\mathcal{F}}).
\end{align*}
for $\gamma_2^{s}$-a.e. $y\in W$. It is clear that also $(\Psi_1)_*(\gamma_1\big|_{\mathcal{U}'_y}\times \nu^{\mathcal{F}}\times \lambda^{\mathcal{F}})\prec \gamma_1^{s}$, so that $\gamma_1^{s}\big|_{\mathcal{U}_y}\sim \gamma_1^{s}$ for $\gamma_2^s$-a.e. $y\in W$. Therefore we have that $\gamma_1^s(X_1\setminus \mathcal{U}_y)=0$ for $\gamma_2^s$-a.e. $y\in W$, so that
\begin{align*}
\mu_s(\mathcal{U}\triangle (X_0^{{\mathcal{F}}}\times W))=0.
\end{align*}
Since this is true for every finite subset $\mathcal{F}\subset G$, we conclude that $\mu_s(X\setminus \mathcal{U})=0$. 

 The conclusion of the proof now follows by combining both claims. Assume that $G\actson (X,\mu_t)$ is not dissipative and fix $s<t$. Choose $r$ such that $s<r<t$. 

$\nu\sim \mu_e$: By claim 2 we have that $G\actson (X,\mu_r)$ is conservative. Then by claim 1 we see that $G\actson (X,\mu_s)$ is weakly mixing. 

$\nu\prec \mu_e$: As $\nu\prec \mu_e$, the measures $\mu_e^{t}$ and $\mu_e$ are equivalent. We have that 
\begin{align*}
\frac{d\mu_g^{t}}{d\mu_e^{t}}=\left((1-t)\frac{d\nu}{d\mu_e}+t\frac{d\mu_g}{d\mu_e}\right)\frac{d\mu_e}{d\mu_e^{t}}.
\end{align*}
So if $\sup_{g\in G}|\log d\mu_g/d\mu_e(x)|<+\infty$ for a.e $x\in X_0$, we also have that $$\sup_{g\in G}|\log d\mu_g^{t}/d\mu_e^{t}(x)|<+\infty \text{ for a.e. }x\in X_0.$$
It follows from \cite[Proposition 4.3]{BV20} that $G\actson (X,\mu_t)$ is conservative. Then by claim 1 we have that $G\actson (X,\mu_s)$ is weakly mixing. 
\end{proof}

\begin{remark}\label{rem:lcsc phase transition}
Let $I$ be a countable infinite set and suppose that we are given a family of equivalent probability measures $(\mu_i)_{i\in I}$ on a standard Borel space $X_0$. Let $\nu$ be a probability measure on $X_0$ that is equivalent with all the $\mu_i$. If $G$ is a locally compact second countable group that acts on $I$ such that for each $i \in I$ the stabilizer subgroup $G_i=\{g\in G:g\cdot i=i\}$ is compact, then the pmp generalized Bernoulli action 
\begin{align*}
G\actson \prod_{i\in I}(X_0,\nu), \;\;\;\; (g\cdot x)_{i}=x_{g^{-1}\cdot i}
\end{align*}
is mixing. For $t\in [0,1]$ write 
\begin{align*}
(X,\mu_t)=\prod_{i \in I}(X_0,(1-t)\nu+t\mu_i)
\end{align*}
and assume that the generalized Bernoulli action $G\actson (X,\mu_1)$ is nonsingular.

Since \cite[Theorem 2.3]{SW81} still applies to infinitely recurrent actions of lcsc groups (see \cite[Remark 7.4]{AIM19}), it is straightforward to adapt the proof of claim 1 in the proof of Theorem \ref{thm:phase transition general} to prove that if $G\actson (X,\mu_t)$ is infinitely recurrent, then $G\actson (X,\mu_s)$ is weakly mixing for every $s<t$. Similarly we can adapt the proof of claim 2, using that a factor of an infinitely recurrent action is again infinitely recurrent. Together, this leads to the following phase transition result in the lcsc setting:

Assume that $G_i=\{g\in G:g\cdot i=i\}$ is compact for every $i\in I$ and that $\nu\sim \mu_e$. Then there exists a $t_1\in [0,1]$ such that $G\actson (X,\mu_t)$ is dissipative up to compact stabilizers for every $t>t_1$ and weakly mixing for every $t<t_1$. 
\end{remark}

Recall the following definition from \cite[Definition 4.2]{BKV19}. When $G$ is a countable infinite group and $G\actson (X,\mu)$ is a nonsingular action on a standard probability space, a sequence $(\eta_n)$ of probability measures on $G$ is called \emph{strongly recurrent} for the action $G\actson (X,\mu)$ if 
\begin{align*}
\sum_{h\in G}\eta_n^2(h)\int_{X}\frac{d\mu(x)}{\sum_{k\in G}\eta_n(hk^{-1})\frac{dk^{-1}\mu}{d\mu}(x)}\xrightarrow{n\rightarrow +\infty} 0.
\end{align*}
We say that $G\actson (X,\mu)$ is \emph{strongly conservative} there exists a sequence $(\eta_n)$ of probability measures on $G$ that is strongly recurrent for $G\actson (X,\mu)$. 

\begin{lemma}\label{lem:strongly recurrent}
Let $G\actson (X,\mu)$ and $G\actson (Y,\nu)$ be nonsingular actions of a countable infinite group $G$ on standard probability spaces $(X,\mu)$ and $(Y,\nu)$. Suppose that $\psi\colon (X,\mu)\rightarrow (Y,\nu)$ is a measure preserving $G$-equivariant factor map and that $\eta_n$ is a sequence of probability measures on $G$ that is strongly recurrent for the action $G\actson (X,\mu)$. Then $\eta_n$ is strongly recurrent for the action $G\actson (Y,\nu)$. 
\end{lemma}

\begin{proof}
Let $E\colon L^0(X,[0,+\infty))\rightarrow L^0(Y,[0,+\infty))$ denote the conditional expectation map that is uniquely determined by 
\begin{align*}
\int_Y E(F) H d\nu=\int_X F (H\circ \psi) d\mu
\end{align*}
for all positive measurable functions $F\colon X\rightarrow [0,+\infty)$ and $H\colon Y\rightarrow [0,+\infty)$. Since 
\begin{align*}
\frac{dk^{-1}\nu}{d\nu}=\frac{d\psi_*(k^{-1}\mu)}{d\psi_*\mu}=E\left(\frac{dk^{-1}\mu}{d\mu}\right)
\end{align*}
for every $k\in G$, we have that
\begin{align}\label{eq:conditional sum}
\sum_{k\in G}\eta_n(hk^{-1})\frac{dk^{-1}\nu}{d\nu}(y)=E\left(\sum_{k\in G}\eta_n(hk^{-1})\frac{dk^{-1}\mu}{d\mu}\right)(y) \text{ for a.e. } y\in Y.
\end{align}
By Jensen's inequality for conditional expectations, applied to the convex function $t\mapsto 1/t$, we also have that
\begin{align}\label{eq:Jensen}
\frac{1}{E\left(\sum_{k\in G}\eta_n(hk^{-1})\frac{dk^{-1}\mu}{d\mu}\right)(y)}\leq  E\left(\frac{1}{\sum_{k\in G}\eta_n(hk^{-1})\frac{dk^{-1}\mu}{d\mu}}\right)(y) \text{ for a.e. } y\in Y.
\end{align}
Combining \eqref{eq:conditional sum} and \eqref{eq:Jensen} we see that 
\begin{align*}
\sum_{h\in G}\eta_n^2(h)\int_{Y}\frac{d\nu(y)}{\sum_{k\in G}\eta_n(hk^{-1})\frac{dk^{-1}\nu}{d\nu}(y)}&\leq \sum_{h\in G}\eta_n^2(h)\int_Y E\left(\frac{1}{\sum_{k\in G}\eta_n(hk^{-1})\frac{dk^{-1}\mu}{d\mu}}\right)(y)d\nu(y)\\
&=\sum_{h\in G}\eta_n^2(h)\int_X\frac{d\mu(x)}{\sum_{k\in G}\eta_n(hk^{-1})\frac{dk^{-1}\mu}{d\mu}(x)},
\end{align*}
which converges to $0$ as $\eta_n$ is strongly recurrent for $G\actson (X,\mu)$.
\end{proof}

We say that a nonsingular group action $G\actson (X,\mu)$ has an \emph{invariant mean} if there exists a $G$-invariant linear functional $\vphi\in L^{\infty}(X)^*$. We say that $G\actson (X,\mu)$ is \emph{amenable (in the sense of Zimmer)} if there exists a $G$-equivariant conditional expectation $E\colon L^{\infty}(G\times X)\rightarrow L^{\infty}(X)$, where the action $G\actson G\times X$ is given by $g\cdot (h,x)=(gh,g\cdot x)$. 

\begin{proposition}\label{prop:invariant mean}
Let $G$ be a countable infinite group and let $(\mu_g)_{g\in G }$ be a family of equivalent probability measures on a standard Borel space $X_0$, that is not supported on a single atom. Let $\nu$ be a probability measure on $X_0$ and for each $t\in [0,1]$ consider the Bernoulli action \eqref{eq:G actson (X,mu_t)}. Assume that $G\actson (X,\mu_1)$ is nonsingular. 
\begin{enumerate}
\item If $G\actson (X,\mu_t)$ has an invariant mean, then $G\actson (X,\mu_s)$ has an invariant mean for every $s<t$.
\item If $G\actson (X,\mu_t)$ is amenable, then $G\actson (X,\mu_s)$ is amenable for every $s>t$.
\item If $G\actson (X,\mu_t)$ is strongly conservative, then $G\actson (X,\mu_s)$ is strongly conservative for every $s<t$.
\end{enumerate}
\end{proposition}

\begin{proof}
1. We may assume that $t=1$. So suppose that $G\actson (X,\mu_1)$ has an invariant mean and fix $s<1$. Let $\lambda$ be the probability measure on $\{0,1\}$ that is given by $\lambda(0)=s$. Then, by \cite[Proposition A.9]{AIM19} the diagonal action $G\actson (X\times X\times \{0,1\}^{G},\mu_1\times \mu_0\times \lambda^{G})$ has an invariant mean. Since $G\actson (X,\mu_s)$ is a factor of this diagonal action, it admits a $G$-invariant mean as well.

2. It suffices to show that $G\actson (X,\mu_1)$ is amenable whenever there exists a $t\in (0,1)$ such that $G\actson (X,\mu_t)$ is amenable. Write $\lambda$ for the probability measure on $\{0,1\}$ given by $\lambda(0)=t$. Then $G\actson (X,\mu_t)$ is a factor of the diagonal action $G\actson (X\times X\times \{0,1\}^{G},\mu_1\times \mu_0\times \lambda^{G})$, so by \cite[Theorem 2.4]{Zim78} also the latter action is amenable. Since $G\actson (X\times \{0,1\}^{G},\mu_0\times \lambda^{G})$ is pmp, we have that $G\actson (X,\mu_1)$ is amenable.

3. We may again assume that $t=1$. Suppose that $(\eta_n)$ is a strongly recurrent sequence of probability measures on $G$ for the action $G\actson (X,\mu_1)$.  Fix $s<1$ and let $\lambda$ be the probability measure on $\{0,1\}$ defined by $\lambda(0)=s$. As the diagonal action $G\actson (X\times \{0,1\}^{G},\mu_0\times \lambda^{G})$ is pmp, the sequence $\eta_n$ is also strongly recurrent for the diagonal action $G\actson (X\times X\times \{0,1\},\mu_1\times \mu_0\times \lambda^{G})$. Since $G\actson (X,\mu_t)$ is a factor of $G\actson (X\times X\times \{0,1\}^{G},\mu_1\times \mu_0\times \lambda^{G})$, it follows from Lemma \ref{lem:strongly recurrent} that the sequence $\eta_n$ is strongly recurrent for $G\actson (X,\mu_t)$.
\end{proof}

We finally prove Theorem \ref{thm:strong ergodicity general}. The proof relies heavily upon the techniques developed in \cite[Section 5]{MV20}.

\begin{proof}[Proof of Theorem \ref{thm:strong ergodicity general}]
For every $t\in (0,1]$ write $\rho^t$ for the Koopman representation
\begin{align*}
\rho^t\colon G\actson L^2(X,\mu_t):\;\;\;\; (\rho^t_g(\xi))(x)=\left(\frac{dg\mu_t}{d\mu_t}(x)\right)^{1/2}\xi(g^{-1}\cdot x).
\end{align*}
Fix $s\in (0,1)$ and let $C>0$ be such that $\log(1-x)\geq -C x$ for every $x\in [0,s)$. Then for every $t<s$ and every $g\in G$ we have that 
\begin{align*}
\log(\langle \rho^t_g(1), 1\rangle)&=\sum_{h\in G}\log(1-H^2(\mu_{gh}^{t},\mu_h^{t}))\\
&\geq \sum_{h\in G}\log(1-tH^2(\mu_{gh},\mu_h))\\
&\geq -C t\sum_{h\in G}H^2(\mu_{gh},\mu_h).
\end{align*}
Because $G\actson (X,\mu_1)$ is nonsingular we get that 
\begin{align}\label{eq:almost invariant rep}
\langle \rho^t_g(1),1\rangle\rightarrow 1 \text{ as } t\rightarrow 0, \text{ for every } g\in G.
\end{align}
We claim that there exists a $t'>0$ such that $G\actson (X,\mu_t)$ is nonamenable for every $t<t'$. Suppose, on the contrary, that $t_n$ is a sequence that converges to zero such that $G\actson (X,\mu_{t_n})$ is amenable for every $n\in \mathbb{N}$. Then it follows from \cite[Theorem 3.7]{Nev03} that $\rho^{t_n}$ is weakly contained in the left regular representation $\lambda_G$ for every $n\in \mathbb{N}$. Write $1_G$ for the trivial representation of $G$. It follows from \eqref{eq:almost invariant rep} that $\bigoplus_{n\in \mathbb{N}}\rho^{t_n}$ has almost invariant vectors, so that
\begin{align*}
1_G\prec \bigoplus_{n\in \mathbb{N}}\rho^{t_n}\prec \infty \lambda_G\prec \lambda_G,
\end{align*}
which is in contradiction with the nonamenability of $G$. By Theorem \ref{thm:phase transition general} there exists a $t_1\in [0,1]$ such that $G\actson (X,\mu_t)$ is weakly mixing for every $t<t_1$. Since every dissipative action is amenable (see for example \cite[Theorem A.29]{AIM19}) it follows that $t_1\geq t'>0$.

Write $Z_0=[0,1)$ and let $\lambda$ denote the Lebesgue probability measure on $Z_0$. Let $\rho^0$ denote the reduced Koopman representation
\begin{align*}
\rho^0\colon G\actson L^2(X\times Z_0^{G},\mu_0\times \lambda^G)\ominus \C 1: \;\;\;\; (\rho^0_g(\xi))(x)=\xi(g^{-1}\cdot x).
\end{align*}
As $G$ is nonamenable, $\rho^{0}$ has stable spectral gap. Suppose that for every $s>0$ we can find $0<s'<s$ such that $\rho^{s'}$ is weakly contained in $\rho^{s'}\otimes \rho^{0}$. Then there exists a sequence $s_n$ that converges to zero, such that $\rho^{s_n}$ is weakly contained in $\rho^{s_n}\otimes \rho^{0}$ for every $n\in \mathbb{N}$. This implies that $\bigoplus_{n\in \mathbb{N}}\rho^{s_n}$ is weakly contained in $(\bigoplus_{n\in \mathbb{N}}\rho^{s_n})\otimes \rho^{0}$. But by \eqref{eq:almost invariant rep}, the representation $\bigoplus_{n\in \mathbb{N}}\rho^{s_n}$ has almost invariant vectors, so that $(\bigoplus_{n\in \mathbb{N}}\rho^{s_n})\otimes \rho^{0}$ weakly contains the trivial representation. This is in contradiction with $\rho^{0}$ having stable spectral gap. We conclude that there exists an $s>0$ such that $\rho^t$ is not weakly contained in $\rho^t\otimes \rho^0$ for every $t<s$. 

We prove that $G\actson (X,\mu_t)$ is strongly ergodic for every $t<\min\{t',s\}$, in which case we can apply \cite[Lemma 5.2]{MV20} to the nonsingular action $G\actson (X,\mu_t)$ and the pmp action $G\actson (X\times Z_0^{G},\mu_0\times \lambda^G)$ by our choice of $t'$ and $s$. After rescaling, we may assume that $G\actson (X,\mu_1)$ is ergodic and that $\rho^{t}$ is not weakly contained in $\rho^{t}\otimes \rho^{0}$ for every $t\in (0,1)$. 

Let $t\in (0,1)$ be arbitrary and define the map
\begin{align*}
\Psi\colon X\times X\times Z_0^{G}\rightarrow X:\;\;\;\; \Psi(x,y,z)_h=\begin{cases}x_h &\text{ if } z_h\leq t\\
y_h &\text{ if } z_h>t\end{cases}.
\end{align*}
Then $\Psi$ is $G$-equivariant and we have that $\Psi(\mu_1\times \mu_0\times \lambda^{G})=\mu_t$. Suppose that $G\actson (X,\mu_t)$ is not strongly ergodic. Then we can find a bounded almost invariant sequence $f_n\in L^{\infty}(X,\mu_t)$ such that $\|f_n\|_2=1$ and $\mu_t(f_n)=0$ for every $n\in \mathbb{N}$. Therefore $\Psi_*(f_n)$ is a bounded almost invariant sequence for $G\actson (X\times X\times Z_0^{G},\mu_1\times \mu_0\times \lambda^{G})$. Let $E\colon L^{\infty}(X\times X\times Z_0^{G})\rightarrow L^{\infty}(X)$ be the conditional expectation that is uniquely determined by $\mu_1\circ E=\mu_1\times \mu_0\times \lambda^{G}$. By \cite[Lemma 5.2]{MV20} we have that $\lim_{n\rightarrow \infty}\|(E\circ \Psi_*)(f_n)-\Psi_*(f_n)\|_2=0$. As $\Psi$ is measure preserving we get in particular that
\begin{align}\label{eq:E circ Psi}
\lim_{n\rightarrow \infty}\|(E\circ \Psi_*)(f_n)\|_2=1.
\end{align}
Note that if $\mu_t(f)=0$ for some $f\in L^{2}(X,\mu_t)$, we have that $\mu_1((E\circ\Psi_*)(f))=0$. So we can view $E\circ \Psi_*$ as a bounded operator
\begin{align*}
E\circ\Psi_*\colon L^2(X,\mu_t)\ominus \C 1\rightarrow L^2(X,\mu_1)\ominus \C 1.
\end{align*}

\vspace{.5cm}

{\bf Claim.} The bounded operator $E\circ\Psi_*\colon L^2(X,\mu_t)\ominus \C 1\rightarrow L^2(X,\mu_1)\ominus \C 1$ has norm strictly less than $1$.

\vspace{.5cm}

The claim is in direct contradiction with \eqref{eq:E circ Psi}, so we conclude that $G\actson(X,\mu_t)$ is strongly ergodic. 

{\it Proof of claim.} For every $g\in G$, let $\vphi_g$ be the map
\begin{align*}
\vphi_g\colon L^2(X_0,\mu_g^t)\rightarrow L^2(X_0,\mu_g): \;\;\;\; \vphi_g(F)=tF+(1-t)\nu(F)\cdot 1.
\end{align*}
Then $E\circ \Psi_*\colon L^2(X_0,\mu_t)\rightarrow L^2(X,\mu_1)$ is given by the infinite product $\bigotimes_{g\in G}\vphi_g$. For every $g\in G$ we have that
\begin{align*}
\|F\|_{2,\mu_g}=\|(d\mu_g^t/d\mu_g)^{-1/2}F\|_{2,\mu_g^t}\leq t^{-1/2}\|F\|_{2,\mu_g^t},
\end{align*}
so that the inclusion map $\iota_g \colon  L^2(X_0,\mu_g^t)\hookrightarrow L^2(X_0,\mu_g)$ satisfies $\|\iota_g\|\leq t^{-1/2}$ for every $g\in G$. We have that
\begin{align*}
\vphi_g(F)=t(F-\mu_g(F)\cdot 1)+\mu_t(F)\cdot 1, \text{ for every } F\in L^2(X_0,\mu_g^t).
\end{align*} 
So if we write $P_g^t$ for the projection map onto $L^2(X_0,\mu_g^t)\ominus \C 1$, and $P_g$ for the projection map onto $L^2(X_0,\mu_g)\ominus \C 1$, we have that 
\begin{align}\label{eq:phi commutes with projection}
\vphi_g\circ P^t_g=t(P_g\circ \iota_g), \text{ for every } g\in G.
\end{align}
For a nonempty finite subset $\mathcal{F}\subset G$ let $V(\mathcal{F})$ be the linear subspace of $L^2(X,\mu_t)\ominus \C 1$ spanned by 
\begin{align*}
\left(\bigotimes_{g\in \mathcal{F}}L^2(X_0,\mu_g^t)\ominus \C 1\right)\otimes \bigotimes_{g\in G\setminus \mathcal{F}}1.
\end{align*}
Then using \eqref{eq:phi commutes with projection} we see that 
\begin{align*}
\|(E\circ \Psi_*)(f)\|_2\leq t^{|\mathcal{F}|/2}\|f\|_2, \text{ for every } f\in V(\mathcal{F}).
\end{align*}
Since $\bigoplus_{\mathcal{F}\neq \emptyset}V(\mathcal{F})$ is dense inside $L^2(X,\mu_t)\ominus \C 1$, we have that
\begin{align*}
\left\|(E\circ\Psi_*)\big|_{L^2(X,\mu_t)\ominus\C 1}\right \|\leq t^{1/2}<1.
\end{align*}
\end{proof}

%% file: trees.tex
\section{Nonsingular Bernoulli actions arising from groups acting on trees: proof of Theorem \ref{thm:trees}}\label{sec:trees}

Let $T$ be a locally finite tree and choose a root $\rho\in T$. Let $\mu_0$ and $\mu_1$ be equivalent probability measures on a standard Borel space $X_0$. Following \cite[Section 10]{AIM19} we define a family of equivalent probability measures $(\mu_e)_{e\in E}$ by
\begin{align}\label{eq:directed family}
\mu_e=\begin{cases}\mu_0 &\text{ if } e \text{ is oriented towards } \rho\\
\mu_1 &\text{ if } e \text{ is oriented away from } \rho\end{cases}.
\end{align}
Let $G\subset \Aut(T)$ be a subgroup. When $g\in G$ and $e\in E$, the edges $e$ and $g\cdot e$ are simultaneously oriented towards, or away from $\rho$, unless $e\in E([\rho,g\cdot \rho])$. As $E([\rho,g\cdot \rho])$ is finite for every $g\in G$, the generalized Bernoulli action 
\begin{align}\label{eq:directed generalized Bernoulli}
G\actson (X,\mu)=\prod_{e\in E}(X_0,\mu_e): \;\;\;\; (g\cdot x)_e=x_{g^{-1}\cdot e}
\end{align}
is nonsingular. If we start with a different root $\rho'\in T$, let $(\mu'_e)_{e\in E}$ denote the corresponding family of probability measures on $X_0$. Then we have that $\mu_e=\mu'_e$ for all but finitely many $e\in E$, so that the measures $\prod_{e\in E}\mu_e$ and $\prod_{e\in E}\mu'_e$ are equivalent. Therefore, up to conjugacy, the action \eqref{eq:directed generalized Bernoulli} is independent of the choice of root $\rho\in T$. 

\begin{lemma}
Let $T$ be a locally finite tree such that each vertex $v\in V(T)$ has degree at least $2$. Suppose that $G\subset \Aut(T)$ is a countable subgroup. Let $\mu_0$ and $\mu_1$ be equivalent probability measures on a standard Borel space $X_0$ and fix a root $\rho \in T$. Then the action $\alpha\colon G\actson (X,\mu)$ given by \eqref{eq:directed generalized Bernoulli} is essentially free. 
\end{lemma}

\begin{proof}
Take $g\in G\setminus\{e\}$. It suffices to show that $\mu(\{x\in X:g\cdot x=x\})=0$. If $g$ is elliptic, there exist disjoint infinite subtrees $T_1,T_2\subset T$ such that $g\cdot T_1=T_2$. Note that 
\begin{align*}
(X_1,\mu_1)=\prod_{e\in E(T_1)}(X_0,\mu_e) \text{ and } (X_2,\mu_2)=\prod_{e\in E(T_2)}(X_0,\mu_e)
\end{align*}
are nonatomic and that $g$ induces a nonsingular isomorphism $\vphi\colon(X_1,\mu_1)\rightarrow (X_2,\mu_2): \vphi(x)_e=x_{g^{-1}\cdot e}$. We get that 
\begin{align*}
\mu_1\times \mu_2(\{(x, \vphi(x)): x\in X_1\})=0.
\end{align*}
A fortiori $\mu(\{x\in X:g\cdot x=x\})=0$. If $g$ is hyperbolic, let $L_g\subset T$ denote its axis on which it acts by nontrivial translation. Then $\prod_{e\in E(L_g)}(X_0,\mu_e)$ is nonatomic and by \cite[Lemma 2.2]{BKV19} the action $g^{\Z}\actson \prod_{e\in E(L_g)}(X_0,\mu_e)$ is essentially free. This implies that also $\mu(\{x\in X:g\cdot x=x\})=0$. 
\end{proof}

We prove Theorem \ref{thm:directed trees} below, which implies Theorem \ref{thm:trees} and also describes the stable type when the action is weakly mixing. 

\begin{theorem}\label{thm:directed trees}
Let $T$ be a locally finite tree with root $\rho\in T$. Let $G\subset \Aut(T)$ be a closed nonelementary subgroup with Poincaré exponent $\delta=\delta(G\actson T)$ given by \eqref{eq:Poincare exponent intro}. Let $\mu_0$ and $\mu_1$ be nontrivial equivalent probability measures on a standard Borel space $X_0$. Consider the generalized nonsingular Bernoulli action $\alpha\colon G\actson (X,\mu)$ given by \eqref{eq:directed generalized Bernoulli}. Then $\alpha$ is 
\begin{itemize}
\item Weakly mixing if $1-H^2(\mu_0,\mu_1)>\exp(-\delta/2)$.
\item Dissipative up to compact stabilizers if $1-H^2(\mu_0,\mu_1)<\exp(-\delta/2)$.
\end{itemize}
Let $G\actson (Y,\nu)$ be an ergodic pmp action and let $\Lambda\subset \R$ be the smallest closed subgroup that contains the essential range of the map 
\begin{align*}
X_0\times X_0\rightarrow \R: \;\;\;\; (x,x')\mapsto \log(d\mu_0/d\mu_1)(x)-\log(d\mu_0/d\mu_1)(x').
\end{align*}
Let $\Delta\colon G\rightarrow \R_{>0}$ denote the modular function and let $\Sigma$ be the smallest subgroup generated by $\Lambda$ and $\log(\Delta(G))$. 

Suppose that $1-H^2(\mu_0,\mu_1)>\exp(-\delta/2)$. Then the Krieger flow and the flow of weights of  $\beta\colon G\actson X\times Y$ are determined by $\Lambda$ and $\Sigma$ as follows. 
\begin{enumerate}
\item If $\Lambda$, resp. $\Sigma$, is trivial, then the Krieger flow, resp. flow of weights, is given by $\R\actson \R$.
\item If $\Lambda$, resp. $\Sigma$, is dense, then the Krieger flow, resp. flow of weights, is trivial. 
\item If $\Lambda$, resp. $\Sigma$, equals $a\Z$, with $a>0$, then the Krieger flow, resp. flow of weights, is given by $\R\actson \R/a\Z$.  
\end{enumerate}
\end{theorem}

In general, we do not know the behaviour of the action \eqref{eq:directed generalized Bernoulli} in the critical situation $1-H^2(\mu_0,\mu_1)=\exp(-\delta/2)$. However, if $T$ is a regular tree and $G\actson T$ has full Poincaré exponent, we prove in Proposition \ref{prop:critical value} below that the action is dissipative up to compact stabilizers. This is similar to \cite[Theorem 8.4 \& Theorem 9.10]{AIM19}.

\begin{proposition}\label{prop:critical value}
Let $T$ be a $q$-regular tree with root $\rho \in T$ and let $G\subset \Aut(T)$ be a closed subgroup with Poincaré exponent $\delta=\delta(G\actson T)=\log(q-1)$. Let $\mu_0$ and $\mu_1$ be equivalent probability measures on a standard Borel space $X_0$. 

If $1-H^2(\mu_0,\mu_1)=(q-1)^{-1/2}$, then the action \eqref{eq:directed generalized Bernoulli} is dissipative up to compact stabilizers. 
\end{proposition}

Interesting examples of actions of the form \eqref{eq:directed generalized Bernoulli} arise when $G\subset \Aut(T)$ is the free group on a finite set of generators acting on its Cayley tree. In that case, following \cite[Section 6]{AIM19} and \cite[Remark 5.3]{MV20}, we can also give a sufficient criterion for strong ergodicity. 

\begin{proposition}\label{prop:free group}
Let the free group $\mathbb{F}_d$ on $d\geq 2$ generators act on its Cayley tree $T$. Let $\mu_0$ and $\mu_1$ be equivalent probability measures on a standard Borel space $X_0$. Then the action \eqref{eq:directed generalized Bernoulli} dissipative if $1-H^2(\mu_0,\mu_1)\leq (2d-1)^{-1/2}$ and weakly mixing and nonamenable if $1-H^2(\mu_0,\mu_1)>(2d-1)^{-1/2}$. Furthermore the action \eqref{eq:directed generalized Bernoulli} is strongly ergodic when $1-H^2(\mu_0,\mu_1)>(2d-1)^{-1/4}$.
\end{proposition}

The proof of Theorem \ref{thm:directed trees} below is similar to that of \cite[Theorem 4]{LP92} and \cite[Theorems 10.3 \& 10.4]{AIM19}

\begin{proof}[Proof of Theorem \ref{thm:directed trees}]
Define a family $(X_e)_{e\in E}$ of independent random variables on $(X,\mu)=\prod_{e\in E}(X_0,\mu_e)$ by
\begin{align}\label{eq:family of rvs}
X_e(x)=\begin{cases}\log (d\mu_1/d\mu_0) (x_e)&\text{ if } e \text{ is oriented towards } \rho\\
\log (d\mu_0/d\mu_1)(x_e) &\text{ if } e \text{ is oriented away from } \rho\end{cases}.
\end{align}
For $v\in T$ we write 
\begin{align*}
S_v=\sum_{e\in E([\rho, v])} X_e.
\end{align*}
Then we have that
\begin{align*}
\frac{dg\mu}{d\mu}=\exp(S_{g\cdot \rho}), \text{ for every  }g\in G.
\end{align*}
Since $G\subset \Aut(T)$ is a closed subgroup, for each $v\in T$, the stabilizer subgroup $G_v=\{g\in G:g\cdot v= v\}$ is a compact open subgroup of $G$. 

Suppose that $1-H^2(\mu_0,\mu_1)<\exp(-\delta/2)$. Then we have that
\begin{align*}
\int_X\sum_{v\in G\cdot \rho}\exp(S_v(x)/2)d\mu(x)=\sum_{v\in G\cdot \rho}(1-H^2(\mu_0,\mu_1))^{2d(\rho,v)}<+\infty,
\end{align*}
by definition of the Poincaré exponent. Therefore we have that $\sum_{v\in G\cdot \rho}\exp(S_v(x)/2)<+\infty$ for a.e. $x\in X$. Let $\lambda$ denote the left invariant Haar measure on $G$ and define $L=\lambda(G_\rho)$, where $G_\rho=\{g\in G:g\cdot \rho=\rho\}$. Then we have that 
\begin{align*}
\int_{G}\frac{dg\mu}{d\mu}(x)d\lambda(g)=L\sum_{v\in G\cdot \rho }\exp(S_v(x))<+\infty, \text{ for a.e. } x\in X.
\end{align*}
We conclude that $G\actson (X,\mu)$ is dissipative up to compact stabilizers.

Now assume that $1-H^2(\mu_0,\mu_1)>\exp(-\delta/2)$. We start by proving that $G\actson (X,\mu)$ is infinitely recurrent. By \cite[Theorem 8.17]{AIM19} we can find a nonelementary closed compactly generated subgroup $G'\subset G$ such that $1-H^2(\mu_0,\mu_1)>\exp(-\delta(G')/2)$. Let $T'\subset T$ be the unique minimal $G'$-invariant subtree. Then $G'$ acts cocompactly on $T'$ and we have that $\delta(G')=\dim_{H}\partial T'$. Let $X$ and $Y$ be independent random variables with distributions $(\log d\mu_1/d\mu_0)_*\mu_0$ and $(\log d\mu_0/d\mu_1)_*\mu_1$ respectively. Set $Z=X+Y$ and write 
\begin{align*}
\vphi(t)=\mathbb{E}(\exp(tZ)).
\end{align*}
The assignment $t\mapsto \vphi(t)$ is convex, $\vphi(t)=\vphi(1-t)$ for every $t$ and $\vphi(1/2)=(1-H^2(\mu_0,\mu_1))^2$. We conclude that
\begin{align*}
\inf_{t\geq 0}\vphi(t)=(1-H^2(\mu_0,\mu_1))^{2}.
\end{align*}
Write $R_k$ for the sum of $k$ independent copies of $Z$. By the Chernoff--Cramér Theorem, as stated in \cite{LP92}, there exists an $M\in \mathbb{N}$ such that
\begin{align}\label{eq:large deviation estimate}
\mathbb{P}(R_M\geq 0)>\exp(-M\delta(G')).
\end{align} 

Below we define a new \emph{unoriented} tree $S$. This means that the edge set of $S$ consists of subsets $\{v,w\}\subset V(S)$. Fix a vertex $\rho'\in T'$ and define the unoriented tree $S$ as follows. 
\begin{itemize}
\item $S$ has vertices $v\in T'$ so that $d_{T'}(\rho', v)$ is divisible by $M$.
\item There is an edge $\{v,w\}\in E(S)$ between two vertices $v,w\in S$ if $d_{T'}(v,w)=M$ and $[\rho',v]_{T'}\subset [\rho',w]_{T'}$. 
\end{itemize}
Here the notation $[\rho',v]_{T'}$ means that we consider the line segment $[\rho',v]$ as a subtree of $T'$. We have that $\dim_H\partial S=M\dim_H \partial T'= M\delta(G')$. Form a random subgraph $S(x)$ of $S$ by deleting those edges $\{v,w\}\in E(S)$ where
\begin{align*}
\sum_{e\in E([v,w]_{T'})}X_e(x_e)<0.
\end{align*}
This is an edge percolation on $S$, where each edge remains with probability $p=\mathbb{P}(R_M\geq 0)$. So by \eqref{eq:large deviation estimate} we have that $p\exp(\dim_H S)>1$. Furthermore if $\{v,w\}$ and $\{v',w'\}$ are edges of $S$ so that $E([v,w]_{T'})\cap E([v',w']_{T'})=\emptyset$, their presence in $S(x)$ are independent events. So the percolation process is a quasi-Bernoulli percolation as introduced in \cite{Lyo89}. Taking $w\in (1,p\exp(\dim_H S))$ and setting $w_n=w^{-n}$, it follows from \cite[Theorem 3.1]{Lyo89} that percolation occurs almost surely, i.e. $S(x)$ contains an infinite connected component for a.e. $x\in X$. Writing 
\begin{align*}
S'_v(x)=\sum_{e\in E([\rho',v]_{T'})}X_e(x_e),
\end{align*}
this means that for a.e. $x\in (X,\mu)$ we can find a constant $a_x>-\infty$ such that $S'_v(x)>a_x$ for infinitely many $v\in T'$. As $T'/G'$ is finite, there exists a vertex $w\in T'$ such that 
\begin{align}\label{eq:divergence w}
\sum_{v\in G'\cdot w}\exp(S'_v(x))=+\infty \;\;\text{ with positive probability.}
\end{align} 
Therefore, by Kolmogorov's zero-one law, we have that $\sum_{v\in G'\cdot w}\exp(S'_v(x))=+\infty$ almost surely. Since a change of root results in a conjugate action, we may assume that $\rho=w$. Then \eqref{eq:divergence w} implies that 
$\sum_{v\in G\cdot \rho}\exp(S_v(x))=+\infty$ for a.e. $x\in X$. Writing again $L$ for the Haar measure of the stabilizer subgroup $G_\rho=\{g\in G:g\cdot \rho= \rho\}$, we see that
\begin{align*}
\int_{G}\frac{dg\mu}{d\mu}d\lambda(g)=L\sum_{v\in G\cdot \rho}\exp(S_v)=+\infty \text{ almost surely.}
\end{align*}
We conclude that $G\actson (X,\mu)$ is infinitely recurrent. We prove that $G\actson (X,\mu)$ is weakly mixing using a phase transition result from the previous section. Define the measurable map
\begin{align*}
\psi\colon X_0\rightarrow (0,1]:\;\;\;\; \psi(x)=\min\{d\mu_1/d\mu_0(x),1\}.
\end{align*}
Let $\nu$ be the probability measure on $X_0$ determined by 
\begin{align*}
\frac{d\nu}{d\mu_0}(x)=\rho^{-1}\psi(x),\;\;\;\; \text{where } \rho=\int_{X_0}\psi(x)d\mu_0(x).
\end{align*}
Then we have that $\nu\sim \mu_0$ and for every $s>1-\rho$ the probability measures 
\begin{align*}
\eta_0^s&=s^{-1}(\mu_0-(1-s)\nu)\\
\eta_1^s&=s^{-1}(\mu_1-(1-s)\nu)
\end{align*}
are well-defined. We consider the nonsingular actions $G\actson (X,\eta_s)=\prod_{e\in E}(X_0,\eta_e^s)$, where 
\begin{align*}
\eta_e^s=\begin{cases}\eta_0^s &\text{ if } e \text{ is oriented towards } \rho\\
\eta_1^s &\text{ if } e \text{ is oriented away from } \rho\end{cases}.
\end{align*}
By the dominated convergence theorem we have that $H^2(\eta_0^s,\eta_1^s)\rightarrow H^2(\mu_0,\mu_1)$ as $s\rightarrow 1$. So we can choose $s$ close enough to $1$, but not equal to $1$, such that $1-H^2(\eta_0^s,\eta_1^s)>\exp(-\delta/2)$. By the first part of the proof we have that $G\actson (X,\eta_s)$ is infinitely recurrent. Note that 
\begin{align*}
\mu_j=(1-s)\nu+s\eta_j^s, \text{ for } j=0,1.
\end{align*}
Since we assumed that $G\subset \Aut(T)$ is closed, all the stabilizer subgroups $G_{v}=\{g\in G:g\cdot v=v\}$ are compact. By Remark \ref{rem:lcsc phase transition} we conclude that $G\actson (X,\mu)$ is weakly mixing. 

Let $G\actson (Y,\nu)$ be an ergodic pmp action. To determine the Krieger flow and the flow of weights of $\beta\colon G\actson X\times Y$ we use a similar approach as in \cite[Theorem 10.4]{AIM19} and \cite[Proposition 7.3]{VW17}. First we determine the Krieger flow and then we deal with the flow of weights.

As before, let $G'\subset G$ be a nonelementary compactly generated subgroup such that $1-H^2(\mu_0,\mu_1)>\exp(-\delta(G')/2)$. By \cite[Theorem 8.7]{AIM19} we may assume that $G/G'$ is not compact. Let $T'\subset T$ be the minimal $G'$-invariant subtree. Let $v\in T'$ be as in Lemma \ref{lem:trivial intersection} below so that
\begin{align}\label{eq:empty intersection}
\bigcap_{g\in G} \left(E(gT')\cup E([v,g^{-1}\cdot v])\right)=\emptyset.
\end{align}
Since changing the root yields a conjugate action, we may assume that $\rho=v$. Let $(Z_0,\zeta_0)$ be a standard probability space such that there exist measurable maps $\theta_0,\theta_1\colon Z_0\rightarrow X_0$ that satisfy $(\theta_0)_*\zeta_0=\mu_0$ and $(\theta_1)_*\zeta_0=\mu_1$. Write 
\begin{align*}
(Z,\zeta)=\prod_{e\in E(T)\setminus E(T')}(Z_0,\zeta_0), \;\;\;\; (X_1,\rho_1)=\prod_{e\in E(T)\setminus E(T')}(X_0,\mu_e), \;\;\;\; (X_2,\rho_2)=\prod_{e\in E(T')}(X_0,\mu_e)
\end{align*}
By the first part of the proof we have that $G'\actson (X_2,\rho_2)$ is infinitely recurrent. Define the probability measure preserving map
\begin{align*}
\Psi\colon (Z,\zeta)\rightarrow (X_1,\rho_1):\;\;\;\; (\Psi(z))_e=\begin{cases}\theta_0(z_e) &\text{ if } e \text{ is oriented towards } \rho\\
\theta_1(z_e) &\text{ if  } e \text{ is oriented away from } \rho\end{cases}.
\end{align*}
Consider 
\begin{align*}
U=\{e\in E(T): e \text{ is oriented towards } \rho\}.
\end{align*}
Since $gU\triangle U=E(T)([\rho,g\cdot \rho])\subset E(T')$ for any $g\in G'$, the set $(E(T)\setminus E(T'))\cap U$ is $G'$-invariant. Therefore $\Psi$ is a $G'$-equivariant factor map. Consider the Maharam extensions
\begin{align*}
G'\actson Z\times X_2\times Y\times \R,\text{ and } G\actson X\times Y\times \R
\end{align*}
of the diagonal actions $G'\actson Z\times X_2\times Y$ and $G'\actson X\times Y\times \R$ respectively. Identifying $(X,\mu)=(X_1,\rho_1)\times (X_2,\rho_2)$ we obtain a $G'$-equivariant factor map 
\begin{align*}
\Phi\colon Z\times X_2\times Y\times \R\rightarrow X_1\times X_2\times Y\times \R:\;\;\;\; \Phi(z,x,y,t)=(\Psi(z),x,y,t).
\end{align*}
Take $F\in L^{\infty}(X\times Y\times \R)^{G}$. By \cite[Proposition A.33]{AIM19} the Maharam extension $G'\actson X_2\times Y\times\R $ is infinitely recurrent. Since $G'\actson Z$ is a mixing pmp generalized Bernoulli action we have that $F\circ \Phi \in L^{\infty}(Z\times X_2\times Y\times \R)^{G}\subset 1\ovt L^{\infty}(X_2\times Y\times \R)^{G}$, by \cite[Theorem 2.3]{SW81}. Therefore $F$ is essentially independent of the $E(T)\setminus E(T')$-coordinates. Thus for any $g\in G$ the assignment 
\begin{align*}
(x,y,t)\mapsto F(g\cdot x,y,t)=F(x,y,t-\log(dg^{-1}\mu/d\mu)(x))
\end{align*}
is essentially independent of the $E(T)\setminus E(gT')$-coordinates. Since $\log(dg^{-1}\mu/d\mu)$ only depends on the $E([\rho,g^{-1}\cdot \rho])$-coordinates, we deduce that $F$ is essentially independent of the $E(T)\setminus(E(gT')\cup E([\rho,g^{-1}\cdot\rho]))$-coordinates, for every $g\in G$. Therefore, by \eqref{eq:empty intersection}, we have that $F\in 1\ovt L^{\infty}(Y\times \R)$. 

So we have proven that any $G$-invariant function $F\in L^{\infty}(X\times Y\times \R)$ is of the form $F(x,y,t)=H(y,t)$, for some $H\in L^{\infty}(Y\times \R)$ that satisfies 
\begin{align*}
H(y,t)=H(g\cdot y, t+\log(dg^{-1}\mu/d\mu)(x)) \text{ for a.e. } (x,y,t)\in X\times Y\times \R.
\end{align*}
Since $0$ is in the essential range of the maps $\log(dg\mu/d\mu)$, for every $g\in G$, we see that $H(g\cdot y,t)=H(y,t)$ for a.e. $(y,t)\in Y\times \R$. By ergodicity of $G\actson Y$, we conclude that $H$ is of the form $H(y,t)=P(t)$, for some $P\in L^{\infty}(\R)$ that satisfies 
\begin{align}\label{eq:invariance P}
P(t)=P(t+\log(dg^{-1}\mu/d\mu)(x))\text{ for a.e. } (x,t)\in X\times \R, \text{ for every } g\in G. 
\end{align}
Let $\Gamma\subset \R$ be the subgroup generated by the essential ranges of the maps $\log(dg\mu/d\mu)$, for $g\in G$. If $\Gamma=\{0\}$ we can identify $L^{\infty}(X\times Y\times \R)^{G}\cong L^{\infty}(\R)$. If $\Gamma\subset \R$ is dense, then it follows that $P$ is essentially constant so that the Maharam extension $G\actson X\times Y\times \R$ is ergodic, i.e. the Krieger flow of $G\actson X\times Y$ is trivial. If $\Gamma=a\Z$, with $a>0$, we conclude by \eqref{eq:invariance P} that we can identify $L^{\infty}(X\times Y\times \R)^{G}\cong L^{\infty}(\R/a\Z)$, so that the Krieger flow of $G\actson X\times Y$ is given by $\R\actson \R/a\Z$. Finally note that the closure of $\Gamma$ equals the closure of the subgroup generated by the essential range of the map
\begin{align*}
X_0\times X_0\rightarrow \R\colon \;\;\;\; (x,x')\mapsto \log(d\mu_0/d\mu_1)(x)-\log(d\mu_0/d\mu_1)(x').
\end{align*}
So we have calculated the Krieger flow in every case, concluding the proof of the theorem in the case $G$ is unimodular. 

When $G$ is not unimodular, let $G_0=\ker \Delta$ be the kernel of the modular function. Let $G\actson X\times Y\times \R$ be the modular Maharam extension and let $\alpha\colon G_0\actson X\times Y\times \R$ be its restriction to the subgroup $G_0$. Then we have that
\begin{align*}
L^{\infty}(X\times Y\times \R)^{G}\subset L^{\infty}(X\times Y \times \R)^{\alpha}.
\end{align*}
By \cite[Theorem 8.16]{AIM19} we have that $\delta(G_0)=\delta$, and we can apply the argument above to conclude that $L^{\infty}(X\times Y\times \R)^{\alpha}\subset 1\ovt 1\ovt L^{\infty}(\R)$. So for every $F\in L^{\infty}(X\times Y\times \R)^{G}$ there exists a $P\in L^{\infty}(\R)$ such that
\begin{align}\label{eq:modular invariance}
P(t)=P(t+\log(dg^{-1}\mu/d\mu)(x)+\log(\Delta(g))) \text{ for a.e. } (x,t)\in X\times \R, \text{ for every } g\in G.
\end{align}
Let $\Pi$ be the subgroup of $\R$ generated by the essential range of the maps 
\begin{align*}
x\mapsto \log(dg^{-1}\mu/d\mu)(x)+\log(\Delta(g)), \text{ with } g\in G. 
\end{align*}
As $0$ is contained in the essential range of $\log(dg^{-1}\mu/d\mu)$, for every $g\in G$, we get that $\log(\Delta(G))\subset \Pi$. Therefore, $\Pi$ also contains the subgroup $\Gamma\subset \R$ defined above. Thus the closure of $\Pi$ equals the closure of $\Sigma$, where $\Sigma\subset \R$ is the subgroup as in the statement of the theorem. From \eqref{eq:modular invariance} we conclude that we may identify $L^{\infty}(X\times Y\times \R)^{G}\cong L^{\infty}(\R)^{\Sigma}$, so that the flow of weights of $G\actson X\times Y$ is as stated in the theorem.
\end{proof}

\begin{lemma}\label{lem:trivial intersection}
Let $T$ be a locally finite tree and let $G\subset \Aut(T)$ be a closed subgroup. Suppose that $H\subset G$ is a closed compactly generated subgroup that contains a hyperbolic element and assume that $G/H$ is not compact. Let $S\subset T$ be the unique minimal $H$-invariant subtree. Then there exists a vertex $v\in S$ such that 
\begin{align}\label{eq:trivial intersection}
\bigcap_{g\in G}\left(gS\cup [v,g^{-1}\cdot v]\right)=\{v\}.
\end{align}
\end{lemma}

\begin{proof}
Let $k\in H$ be a hyperbolic element and let $L\subset T$ be its axis, on which $k$ acts by a nontrivial translation. Then $L\subset S$, as one can show for instance as in the proof of \cite[Proposition 3.8]{CM11}. Pick any vertex $v\in L$. We claim that this vertex will satisfy \eqref{eq:trivial intersection}. Take any $w\in V(T)\setminus\{v\}$. As $G/H$ is not compact, one can show as in \cite[Theorem 9.7]{AIM19} that there exists a $g\in G$ such that $g\cdot w\notin S$. Since $k$ acts by translation on $L$, there exists an $n\in \mathbb{N}$ large enough such that
\begin{align*}
[v,k\cdot v]\subset [v,k^ng\cdot v]\text{ and } [v,k^{-1}\cdot v]\subset [v,k^{-n}g\cdot v],
\end{align*}   
so that in particular we have that $w\notin [v,k^ng\cdot v]\cap [v,k^{-n}g\cdot v]=\{v\}$. Since $S$ is $H$-invariant, we also have that $k^ng\cdot w\notin S$ and $k^{-n}g\cdot w\notin S$ and we conclude that
\begin{align*}
w\notin \left((k^ng)^{-1}S\cup[v,k^ng\cdot v]\right)\cap\left((k^{-n}g)^{-1}S\cup [v,k^{-n}g\cdot v]\right).
\end{align*}
\end{proof}

\begin{proof}[Proof of Proposition \ref{prop:critical value}]
Define the family $(X_e)_{e\in E}$ of independent random variables on $(X,\mu)$ by \eqref{eq:family of rvs} and write
\begin{align*}
S_v=\sum_{e\in E([\rho,v])}X_e.
\end{align*}
{\bf Claim.} There exists a $\delta>0$ such that 
\begin{align*}
\mu(\{x\in X:S_v(x)\leq -\delta \text{ for every } v\in T\setminus \{\rho\}\})>0.
\end{align*}
{\it Proof of claim.} Note that $\mathbb{E}(\exp(X_e/2))=1-H^2(\mu_0,\mu_1)$ for every $e\in E$. Define a family of random variables $(W_n)_{n\geq 0}$ on $(X,\mu)$ by
\begin{align*}
W_n=\sum_{\substack{v\in T\\d(v,\rho)=n}}\exp(S_v/2).
\end{align*}
Using that $1-H^2(\mu_0,\mu_1)=(q-1)^{-1/2}$ one computes that
\begin{align*}
\mathbb{E}(W_{n+1}|\;S_v,\; d(v,\rho)\leq n)=W_n, \text{ for every } n\geq 1.
\end{align*}
So the sequence $(W_n)_{n\geq 0}$ is a martingale and since it is positive, it converges almost surely to a finite limit when $n\rightarrow +\infty$. Write $\Sigma_n=\{v\in T:d(v,\rho)=n\}$. As $W_n\geq \max_{v\in \Sigma_n}\exp(S_v/2)$ we conclude that there exists a positive constant $C<+\infty$ such that
\begin{align*}
\mathbb{P}(S_v\leq C \text{ for every }v\in T)>0.
\end{align*}
For any vertex $w\in T$, write $T_w=\{v\in T:[\rho,w]\subset [\rho,v]\}$: the set of children of $w$, including also $w$ itself. Using the symmetry of the tree and changing the root from $\rho$ to $w\in T$, we also have that
\begin{align}\label{eq:probability1}
\mathbb{P}(S_v-S_w\leq C \text{ for every } v\in T_w)>0, \text{ for every } w\in T.
\end{align} 
Set $\nu_0=(\log d\mu_1/d\mu_0)_*\mu_0$ and $\nu_1=(\log d\mu_0/d\mu_1)_*\mu_1$. Because $1-H^2(\mu_0,\mu_1)\neq 0$ we have that $\mu_0\neq \mu_1$, so that there exists a $\delta>0$ such that
\begin{align*}
\nu_0*\nu_1((-\infty,-\delta))>0.
\end{align*}
Here $\nu_0*\nu_1$ denotes the convolution product of $\nu_0$ with $\nu_1$. Therefore there exists $N\in \mathbb{N}$ large enough such that 
\begin{align}\label{eq:probability2}
\mathbb{P}(S_w\leq -C-\delta \text{ for every } w\in \Sigma_N \text{ and } S_{w'}\leq-\delta \text{ for every } w'\in \Sigma_n \text{ with } n\leq N)>0.
\end{align}
Since for any $w\in \Sigma_N$ and $w'\in \Sigma_n$ with $n\leq N$, we have that $S_v-S_w$ is independent of $S_{w'}$ for every $v\in T_w$, and since $\Sigma_N$ is a finite set, it follows from \eqref{eq:probability1} and \eqref{eq:probability2} that
\begin{align*}
\mathbb{P}(S_v\leq -\delta \text{ for every } v\in T\setminus\{\rho\})>0. 
\end{align*}
This concludes the proof of the claim.

Let $\delta>0$ be as in the claim and define 
\begin{align*}
\mathcal{U}=\{x\in X:S_v(x)\leq -\delta \text{ for every } v\in T\setminus \{\rho\}\},
\end{align*}
so that $\mu(\mathcal{U})>0$. Let $G_\rho$ be the stabilizer subgroup of $\rho$. Note that for every $g,h\in G$ we have that $S_{hg\cdot \rho}(x)=S_{g\cdot \rho}(h^{-1}\cdot x)+S_{h\cdot \rho}(x)$ for a.e. $x\in X$, so that for $h\in G$ we have that
\begin{align*}
h\cdot \mathcal{U}\subset \{x\in X:S_{hg\cdot \rho}(x)\leq -\delta+S_{h\cdot \rho}(x) \text{ for every } g\notin G_{\rho}\}.
\end{align*}
It follows that if $h\notin G_{\rho}$, we have that
\begin{align*}
\mathcal{U}\cap h\cdot \mathcal{U}\subset \{x\in X:S_{h\cdot \rho}(x)\leq -\delta \text{ and } S_{h\cdot \rho}(x)\geq \delta\}=\emptyset.
\end{align*}
Since $G\subset \Aut(T)$ is closed, we have that $G_\rho$ is compact. So the action $G\actson (X,\mu)$ is not infinitely recurrent. Let $\lambda$ denote the left invariant Haar measure on $G$. By an adaptation of the proof of \cite[Proposition 4.3]{BV20}, the set 
\begin{align*}
D=\left\{ x\in X: \int_{G}\frac{dg\mu}{d\mu}(x)d\lambda(g)<+\infty\right\}=\left\{ x\in X: \int_{G}\exp(S_{g\cdot \rho}(x))d\lambda(g)<+\infty\right\}
\end{align*}
satisfies $\mu(D)\in \{0,1\}$. Since $G\actson (X,\mu)$ is not infinitely recurrent, it follows from \cite[Proposition A.28]{AIM19} that $\mu(D)>0$, so that we must have that $\mu(D)=1$. By \cite[Theorem A.29]{AIM19} the action $G\actson (X,\mu)$ is dissipative up to compact stabilizers. 
\end{proof}

We use a similar approach as in \cite[Section 6]{MV20} in the proof of Proposition \ref{prop:free group} below. 

\begin{proof}[Proof of Proposition \ref{prop:free group}]
It follows from Theorem \ref{thm:directed trees} and Proposition \ref{prop:critical value} that the action $G\actson (X,\mu)$, given by \eqref{eq:directed generalized Bernoulli}, is dissipative when $1-H^2(\mu_0,\mu_1)\leq (2d-1)^{-1/2}$ and weakly mixing when $1-H^2(\mu_0,\mu_1)>(2d-1)^{-1/2}$. So it remains to show that $G\actson (X,\mu)$ is nonamenable when $1-H^2(\mu_0,\mu_1)>(2d-1)^{-1/2}$ and strongly ergodic when $1-H^2(\mu_0,\mu_1)>(2d-1)^{-1/4}$.

Assume first that $1-H^2(\mu_0,\mu_1)>(2d-1)^{-1/2}$. By taking the kernel of a surjective homomorphism $\mathbb{F}_d\rightarrow \Z$ we find a normal subgroup $H_1\subset \mathbb{F}_d$ that is free on infinitely many generators. By \cite[Théorème 0.1]{RT13} we have that $\delta(H_1)=(2d-1)^{-1/2}$. Then using \cite[Corollary 6]{Su79}, we can find a finitely generated free subgroup $H_2\subset H_1$ such that $H_1=H_2*H_3$ for some free subgroup $H_3\subset H_1$ and such that $1-H^2(\mu_0,\mu_1)>\exp(-\delta(H_2)/2)$. Let $\psi\colon H_1\rightarrow H_3$ be the surjective group homomorphism uniquely determined by 
\begin{align*}
\psi(h)=\begin{cases}e &\text{ if } h\in H_2\\
h&\text{ if } h\in H_3\end{cases}.
\end{align*}
We set $N=\ker\psi$, so that $H_2\subset N$ and we get that $1-H^2(\mu_0,\mu_1)>\exp(-\delta(N)/2)$. Therefore $N\actson (X,\mu)$ is ergodic by Theorem \ref{thm:directed trees}. Also we have that $H_1/N\cong H_3$, which is a free group on infinitely many generators. Therefore $H_1\actson (X,\mu)$ is nonamenable by \cite[Lemma 6.4]{MV20}. A posteriori also $\mathbb{F}_d\actson(X,\mu)$ is nonamenable. 

Let $\pi$ be the Koopman representation of the action $\mathbb{F}_d\actson (X,\mu)$:
\begin{align*}
\pi\colon G\actson L^2(X,\mu):\;\;\;\; (\pi_g(\xi))(x)=\left(\frac{dg\mu}{d\mu}(x)\right)^{1/2}\xi(g^{-1}\cdot x).
\end{align*}
{\bf Claim.} If $1-H^2(\mu_0,\mu_1)>(2d-1)^{-1/4}$, then $\pi$ is not weakly contained in the left regular representation.
\vspace{.5cm}

{\it Proof of claim.} Let $\eta$ denote the canonical symmetric measure on the generator set of $\mathbb{F}_d$ and define
\begin{align*}
P=\sum_{g\in \mathbb{F}_d}\eta(g)\pi_g.
\end{align*}
The $\eta$-spectral radius of $\alpha\colon \mathbb{F}_d\actson (X,\mu)$, which we denote by $\rho_\eta(\alpha)$, is by definition the norm of $P$, as a bounded operator on $L^2(X,\mu)$. By \cite[Proposition A.11]{AIM19} we have that 
\begin{align*}
\rho_\eta(\alpha)&=\lim_{n\rightarrow \infty}\langle P^n(1), 1\rangle^{1/n}\\
&=\lim_{n\rightarrow \infty}\left(\sum_{g\in \mathbb{F}_d}\eta^{*n}(g)(1-H^2(\mu_0,\mu_1))^{2|g|}\right)^{1/n},
\end{align*}
where $|g|$ denotes the word length of a group element $g\in \mathbb{F}_d$. By \cite[Theorem 6.10]{AIM19} we then have that
\begin{align*}
\rho_\eta(\alpha)=\begin{cases}\frac{(1-H^2(\mu_0,\mu_1))^2}{2d}\left((2d-1)+(1-H^2(\mu_0,\mu_1))^{-4}\right)& \text{ if }1-H^2(\mu_0,\mu_1)>(2d-1)^{-1/4}\\
\frac{\sqrt{2d-1}}{d}&\text{ if } 1-H^2(\mu_0,\mu_1)\leq(2d-1)^{-1/4}\end{cases}.
\end{align*}
Therefore, if $1-H^2(\mu_0,\mu_1)>(2d-1)^{-1/4}$, we have that $\rho_\eta(\alpha)>\rho_\eta(\mathbb{F}_d)$, where $\rho_\eta(\mathbb{F}_d)$ denotes the $\eta$-spectral radius of the left regular representation. This implies that $\alpha$ is not weakly contained in the left regular representation (see for instance \cite[Section 3.2]{AD03}).

Now assume that $1-H^2(\mu_0,\mu_1)>(2d-1)^{-1/4}$. As in the proof of Theorem \ref{thm:directed trees} there exist probability measures $\nu, \eta_0$ and $\eta_1$ on $X_0$ that are equivalent with $\mu_0$ and a number $s\in (0,1)$ such that 
\begin{align*}
\mu_j=(1-s)\nu+s\eta_j, \text{ for } j=0,1,
\end{align*}
and such that $1-H^2(\eta_0,\eta_1)>(2d-1)^{-1/4}$. Consider the nonsingular action
\begin{align*}
\mathbb{F}_d\actson (X,\eta)=\prod_{e\in E(T)}(X_0,\eta_e),\text{ where } \eta_e=\begin{cases}\eta_0&\text{ if } e \text{ is oriented towards } \rho\\
\eta_1 &\text{ if } e \text{ is oriented away from } \rho\end{cases}.
\end{align*}
By Theorem \ref{thm:directed trees} the action $\mathbb{F}_d\actson (X,\eta)$ is ergodic. Write $\rho$ for the Koopman representation associated to $\mathbb{F}_d\actson (X,\eta)$. By the claim, $\rho$ is not weakly contained in the left regular representation. Let $\lambda$ be the probability measure on $\{0,1\}$ given by $\lambda(0)=s$. Let $\rho^0$ be the reduced Koopman representation of the pmp generalized Bernoulli action $\mathbb{F}_d\actson (X\times \{0,1\}^{E(T)},\nu^{E(T)}\times\lambda^{E(T)})$. Then $\rho^0$ is contained in a multiple of the left regular representation. Therefore, as $\rho$ is not weakly contained in the left regular representation, $\rho$ is not weakly contained in $\rho\otimes \rho^{0}$. 

Define the map
\begin{align*}
\Psi\colon X\times X\times \{0,1\}^{E(T)}\rightarrow X\colon \;\;\;\; \Psi(x,y,z)_e=\begin{cases}x_e &\text{ if } z_e=0\\
y_e &\text{ if } z_e=1\end{cases}.
\end{align*}
Then $\Psi$ is $\mathbb{F}_d$-equivariant and we have that $\Psi_*(\eta\times \nu^{E(T)}\times \lambda^{E(T)})=\mu$. Suppose that $\mathbb{F}_d\actson (X,\mu)$ is not strongly ergodic. Then there exists a bounded almost invariant sequence $f_n\in L^{\infty}(X,\mu)$ such that $\|f_n\|_2=1$ and $\mu(f_n)=0$ for every $n\in \mathbb{N}$. Therefore $\Psi_*(f_n)$ is a bounded almost invariant sequence for the diagonal action $\mathbb{F}_d\actson (X\times X\times \{0,1\}^{E(T)},\eta\times \nu^{E(T)}\times \lambda^{E(T)})$. Let $E\colon L^{\infty}(X\times X\times \{0,1\}^{E(T)})\rightarrow L^{\infty}(X)$ be the conditional expectation that is uniquely determined by $\mu\circ E=\eta\times \nu^{E(T)}\times \lambda^{E(T)}$. By \cite[Lemma 5.2]{MV20} we have that $\lim_{n\rightarrow \infty}\|(E\circ\Psi_*)(f_n)-\Psi_{*}(f_n)\|_2=0$ and in particular we get that
\begin{align}\label{eq:limit 1}
\lim_{n\rightarrow \infty}\|(E\circ \Psi_*)(f_n)\|_2=1.
\end{align} 
But just as in the proof of Theorem \ref{thm:strong ergodicity general} we have that 
\begin{align*}
\left\|(E\circ\Psi_*)\big|_{L^2(X,\mu)\ominus \mathbb{C}1}\right\|<1,
\end{align*}
which is in contradiction with \eqref{eq:limit 1}. We conclude that $\mathbb{F}_d\actson (X,\mu)$ is strongly ergodic. 
\end{proof}

Proposition \ref{prop:directed not closed} below complements Theorem \ref{thm:directed trees} by considering groups $G\subset \Aut(T)$ that are not closed. This is similar to \cite[Theorem 10.5]{AIM19}.

\begin{proposition}\label{prop:directed not closed}
Let $T$ be a locally finite tree with root $\rho\in T$. Let $G\subset \Aut(T)$ be a lcsc group such that the inclusion map $G\rightarrow \Aut(T)$ is continuous and such that $G\subset \Aut(T)$ is not closed. Write $\delta=\delta(G\actson T)$ for the Poincaré exponent given by \eqref{eq:Poincare exponent intro}. Let $\mu_0$ and $\mu_1$ be nontrivial equivalent probability measures on a standard Borel space $X_0$. Consider the generalized nonsingular Bernoulli action $\alpha\colon G\actson (X,\mu)$ given by \eqref{eq:directed generalized Bernoulli}. Let $H\subset \Aut(T)$ be the closure of $G$. Then the following holds. 
\begin{itemize}
\item If $1-H^2(\mu_0,\mu_1)>\exp(-\delta/2)$, then $\alpha$ is ergodic and its Krieger flow is determined by the essential range of the map
\begin{align}\label{eq:essential range}
X_0\times X_0\rightarrow \R:\;\;\;\; (x,x')\mapsto\log(d\mu_0/d\mu_1)(x)-\log(d\mu_0/d\mu_1)(x')
\end{align}
as in Theorem \ref{thm:directed trees}.
\item If $1-H^2(\mu_0,\mu_1)<\exp(-\delta/2)$, then each ergodic component of $\alpha$ is of the form $G\actson H/K$, where $K$ is a compact subgroup of $H$. In particular there exists a $G$-invariant $\sigma$-finite measure on $X$ that is equivalent with $\mu$.
\end{itemize}
\end{proposition}

\begin{proof}
Let $H\subset \Aut(T)$ be the closure of $G$. Then $\delta(H)=\delta$ and we can apply Theorem \ref{thm:directed trees} to the nonsingular action $H\actson (X,\mu)$. 

If $1-H^2(\mu_0,\mu_1)>\exp(-\delta/2)$, then $H\actson X$ is ergodic. As $G\subset H$ is dense, we have that
\begin{align*}
L^{\infty}(X)^{G}=L^{\infty}(X)^{H}=\mathbb{C}1,
\end{align*}
so that $G\actson X$ is ergodic. Let $H\actson X\times \R$ be the Maharam extension associated to $H\actson X$. Again, as $G\subset H$ is dense, we have that
\begin{align*}
L^{\infty}(X\times \R)^{G}=L^{\infty}(X\times \R)^{H}.
\end{align*}
Note that the subgroup generated by the essential ranges of the maps $\log(dg^{-1}\mu/d\mu)$, with $g\in G$, is the same as the subgroup generated by the essential ranges of the maps $\log(dh^{-1}\mu/d\mu)$, with $h\in H$. Then one determines the Krieger flow of $G\actson X$ as in the proof of Theorem \ref{thm:directed trees}. 

If $1-H^2(\mu_0,\mu_1)<\exp(-\delta/2)$, the action $H\actson (X,\mu)$ is dissipative up to compact stabilizers. By \cite[Theorem A.29]{AIM19} each ergodic component is of the form $H\actson H/K$ for a compact subgroup $K\subset H$. Therefore each ergodic component of $G\actson (X,\mu)$ is of the form $G\actson H/K$, for some compact subgroup $K\subset H$. 
\end{proof}